\documentclass[10pt, oneside, a4paper]{article}
\usepackage{geometry}
\title{A new family of elliptic curves with unbounded rank}
\author{\large Richard \textsc{Griffon}} 
\date{} 

\usepackage[english]{babel}
\usepackage[T1]{fontenc}
\usepackage[utf8]{inputenc} 
\usepackage{lmodern}
\usepackage{amssymb,amsmath,amscd}  
\usepackage{mathrsfs}
\usepackage{ntheorem}
\usepackage{multicol}
\usepackage{tikz-cd}

{\theoremstyle{plain}
\theoremseparator{ --}
\newtheorem{theo}{Theorem}[section]}
{\theoremstyle{plain}
\theoremseparator{ --}
\newtheorem{coro}[theo]{Corollary}}
{\theoremseparator{ --}
\newtheorem{lemm}[theo]{Lemma}}
{\theoremseparator{ --}
\newtheorem{prop}[theo]{Proposition}}
{\theoremseparator{ --}
}
{\theoremseparator{ --}
\newtheorem{itheo}{Theorem}}

{\theorembodyfont{\normalfont}
\newtheorem{defi}[theo]{Definition}}
{\theoremstyle{plain}
\theorembodyfont{\normalfont}
}
{\theoremstyle{plain}
\theorembodyfont{\normalfont}
\newtheorem{rema}[theo]{Remark}}

{\theoremheaderfont{\itshape}
\theorembodyfont{\normalfont}
\theoremseparator{: }
\newtheorem*{proof}{Proof}}

\newcommand{\ProofEnd}{\hfill \ensuremath{\Box}}

\usepackage{tabularx}
\usepackage{array}
\usepackage{multirow}
\usepackage{multicol}
\usepackage{enumerate}
\usepackage{graphicx}

\usepackage{titlesec}
\titleformat{\subsection}[runin]
        {\large\bfseries}
        {\thesubsection.}
        {0.2em}
        {}
        [.\hspace{0.3em}-- ]
        
 \titleformat{\section}%
        {\center\large\bfseries}
        {\thesection.}
        {0.2em}
        {}
        []

\usepackage{fancyhdr}

\usepackage{url}

\DeclareMathOperator{\rk}{rank}

\DeclareMathOperator{\ord}{ord}
\DeclareMathOperator{\DEGRE}{deg }
\renewcommand{\deg}{\DEGRE}

\newcommand{\Q}{\ensuremath{\mathbb{Q}}}

\newcommand{\Z}{\ensuremath{\mathbb{Z}}}
\renewcommand{\P}{\ensuremath{\mathbb{P}}}
\newcommand{\F}{\ensuremath{\mathbb{F}}}
\newcommand{\Qbar}{\ensuremath{\bar{\mathbb{Q}}}}

\newcommand{\Ecal}{\mathcal{E}}
\newcommand{\Escr}{\mathscr{E}}
\newcommand{\Ss}{\mathcal{S}}

\renewcommand{\bar}[1]{\ensuremath{\overline{#1}}}

\renewcommand{\tt}{\mathbf{t}}
\newcommand{\ie}{\textit{i.e.}{}}
\newcommand{\cf}{\textit{cf.}{}}

\usepackage{bm}
\newcommand{\bbeta}{\hm{\beta}}

\usepackage{bbm}
\newcommand{\trivcar}{\mathbbm{1}}

\newcommand{\ja}{\mathbf{j}}
\newcommand{\Mm}{\mathbf{m}}
\newcommand{\Mn}{\mathbf{n}}

\newcommand{\cchi}{\hm{\chi}}
\newcommand{\norm}{\mathbf{N}}
\newcommand{\gP}{\mathfrak{P}}

\newcommand{\rank}{\varrho}

\newcommand{\cond}{\mathcal{N}}
\newcommand{\tors}{_{\mathrm{tors}}}

\renewcommand{\subset}{\subseteq}

\definecolor{myemph}{rgb}{0.7,0.1,0.1}

\usepackage[colorlinks=true,
citecolor= orange, 
linkcolor=  myemph, 
urlcolor=black!60, 
pagebackref
]{hyperref}

\renewcommand{\O}{\mathcal{O}}

\begin{document}
\pagestyle{fancy}
\maketitle 

\noindent\hfill\rule{7cm}{0.5pt}\hfill\phantom{.}

\paragraph{Abstract --} 
Let $\F_q$ be a finite field of odd characteristic and $K= \mathbb{F}_q(t)$.  
For any integer $d\geq 2$ coprime to $q$, consider the elliptic curve $E_d$ over $K$ defined by
$y^2=x\cdot\big(x^2+t^{2d}\cdot x-4t^{2d}\big)$.
We show that the rank of the Mordell--Weil group $E_d(K)$ is unbounded as $d$ varies. 
The curve $E_d$ satisfies the BSD conjecture, so that its rank equals the order of vanishing of its $L$-function at the central point.
We provide an explicit expression for the $L$-function of $E_d$, and use it to study this order of vanishing in terms of $d$. 

\medskip

\noindent%
 
\smallskip
\noindent 
{\it 2010 Math. Subj. Classification:}  
11G05, 
11M38, 
11G40, 
14G10, 
11L05. 

\smallskip
\noindent
{\it E-mail:} \href{mailto:richard.griffon@unibas.ch}{richard.griffon@unibas.ch}
\medskip

\noindent\hfill\rule{7cm}{0.5pt}\hfill\phantom{.}


\section*{Introduction}
\pdfbookmark[0]{Introduction}{Introduction} 
\addcontentsline{toc}{section}{Introduction}
\setcounter{section}{0}

\paragraph{}
Let $\F_q$ be a finite field of characteristic $p\geq 3$ and $K:=\F_q(t)$. 
It has been known for a while that there are elliptic curves over $K$ with arbitrary large Mordell--Weil ranks. 
The first examples were provided in \cite{ShaTate_Rk} by Shafarevich and Tate:  they construct a sequence $\{E_n\}_{n\geq 1}$ of \emph{isotrivial} elliptic curves over $K$ such that $\rk E_n(K)\to\infty$.
The first instance of a sequence of  \emph{nonisotrivial} elliptic curves with unbounded ranks is due to Ulmer in \cite{Ulmer_LargeRk}.
Ulmer subsequently proved a rather general theorem (see \cite{Ulmer_largeLrank}) which ensures, under a parity condition on the conductor, that certain so-called Kummer families of elliptic curves over~$K$ have unbounded \emph{analytic} ranks.
In parallel, Berger  proposed in  \cite{Berger}  a construction of Kummer families of elliptic curves for which the BSD conjecture holds for each curve in the family: in these cases, Ulmer's `unbounded \emph{analytic} rank' result mentioned above can be translated into a `unbounded \emph{algebraic} rank' theorem.

\paragraph{}
In \cite{Berger}, Berger gives three examples of families of elliptic curves to which her construction applies: 
for all but one of them, she concludes about unboundedness of the algebraic rank using Ulmer's result \cite{Ulmer_largeLrank}. 
In this article, we treat the remaining example (to which \cite{Ulmer_largeLrank} does not apply). 
For any integer $d\geq 1$, consider the elliptic curve $E_d/K$ given by the affine Weierstrass model:
\begin{equation*}\label{ieq.Wmodel}
E_d: \qquad 
Y^2+2t^d\cdot XY -4t^{2d}\cdot Y = X^3 -6t^d\cdot X^2+8t^{2d} \cdot X.	
\end{equation*}

The main goal of this article is to prove that: 
\begin{itheo}
As $d\geq 1$ varies, the ranks of the Mordell--Weil groups $E_d(K)$ are unbounded; \ie{}, 
\label{itheo.unbounded}
\[\limsup_{ d\geq 1}\rk(E_{d}(K)) = +\infty. \]
Further, the average Mordell--Weil rank of $E_d(K)$ is unbounded.
\end{itheo}

\noindent
We prove the two assertions of Theorem \ref{itheo.unbounded} separately (see Theorems \ref{theo.unbounded} and \ref{theo.unbounded.redux}).  
First, 
we will construct several sequences of `special' integers $(d_n)_{n\geq 1}$ for which we show that 
\[\rk\big(E_{d_n}(K)\big)\gg_q d_n/\log d_n.\]
More specifically, we exhibit sequences of even (resp. odd) integers $d_n$'s such that this lower bound holds (see Theorem \ref{theo.unbounded}). 
Secondly, 
we use a result from \cite{PomeShpa} to prove that the average rank of $E_d(K)$ is  unbounded: 
more precisely, we will show (in Theorem \ref{theo.unbounded.redux}) that
\[ \exists \alpha>1/2, \qquad  \frac{1}{x}\sum_{d\leq x}\rk\big(E_{d}(K)\big) \geq x^\alpha \qquad (\text{as } x\to\infty).\]
Of course, the second assertion implies the first one, and our treatment might appear redundant.
However,  the two proofs shed different lights on the behaviour of the sequence $d\mapsto \rk\big(E_{d}(K)\big)$.
Indeed, for the explicitly constructed $(d_n)_{n\geq 1}$ in Theorem \ref{theo.unbounded}, the corresponding curves $E_{d_n}$ have `very large' ranks (as large as allowed by Brumer's bound, \cf{} \cite{Brumer} and our Remark \ref{rema.unbdd}\eqref{item.brumer}).
However, these integers~$d_n$ are very far apart from each other, and they cannot be deemed representative of the `typical' size of $\rk\big(E_{d}(K)\big)$.
The average result  fills in that gap by showing that the rank of $E_d$ is `usually large'; but one then loses the explicit and precise character of the first construction.

The second proof also reveals that the sequence~$\{E_d\}_d$ at hand is quite special: indeed, Brumer has proved that the average rank of elliptic curves over $\F_q(t)$ is bounded (see \cite{Brumer}).

\paragraph{}
Let us now explain the strategy of the proof of Theorem \ref{itheo.unbounded}  as we give the plan of the paper.
In section~\ref{sec.1}, we start by introducing the   elliptic curves $E_d$ and by computing their relevant invariants. 
We then describe the torsion subgroup $E_d(K)\tors$ (Theorem~\ref{theo.tors}), and provide a point of infinite order in~$E_d(K)$ (Corollary~\ref{coro.positive.rk}).
We also explain why the results in \cite{Ulmer_largeLrank} cannot be used here.

Our first step towards the proof of Theorem \ref{itheo.unbounded} will be to give an explicit formula for the $L$-function of~$E_d/K$.
To avoid introducing too many notations here, let us only state the following special case of our result (see Theorem \ref{theo.Lfunc} for the general version):

\begin{itheo}\label{itheo.Lfunc}
Let $\F_q$ be a finite field of odd characteristic. 
For any integer $d\geq 1$ such that $2d\mid q-1$, choose a character $\cchi:\F_q^\times\to\Qbar^\times$ of exact order $2d$.
The $L$-function of the elliptic curve $E_{d}/K$ is given by
\[L(E_d/K, T) = (1-qT)\cdot \prod_{\substack{ 1\leq n \leq 2d-1 \\ n\neq d/2, 3d/2} }\left(1 -  B(n)\cdot T\right), \]
 where, for all $n\in\{1, \dots, 2d-1\}$ with $n\neq d/2, 3d/2$, we let
\[ B(n) := \cchi^{2n}(4)\cdot  \sum_{u\in\F_q}\sum_{v\in\F_q} \cchi^n(u(1-u)) \cchi^{2n + d}(v)\cchi^{-n}(1-v).\]
\end{itheo}

\noindent
The relevant objects are introduced in section \ref{sec.prelim} and the proof of Theorem \ref{theo.Lfunc} is given in section \ref{sec.Lfunc}.
It is based on direct manipulations of character sums related to `counting points' on the various reductions of~$E_d$.
Given that there are very few elliptic curves over $K$ for which the $L$-function is explicitly known, this Theorem may be of independent interest. 

The  expression of $L(E_d/K, T)$ in Theorem \ref{itheo.Lfunc} is sufficiently explicit that one can study its order of vanishing at $T=q^{-1}$ (see Corollary \ref{coro.expr.rk}). 
Our main interest in doing so is that the BSD conjecture is known to hold for $E_d/K$ \ie{}, one has:
\[\ord_{T=q^{-1}}L(E_d/K, T) = \rk\big(E_d(K)\big).\]
This fact has been proved by Berger in \cite{Berger}, and we 
sketch her proof in section \ref{sec.construction}.
To this effect, we also briefly recall there how the curves $E_d$ are constructed. 

The expression for $\ord_{T=q^{-1}}L(E_d/K, T)$ obtained in Corollary \ref{coro.expr.rk} becomes more tractable for certain values of $d$: 
specifically, for those $d\geq 1$ such that $2d$ is \emph{supersingular}. 
We give more details about this in section \ref{sec.supersing}, recalling results from \cite{ShaTate_Rk} and \cite{Ulmer_LargeRk}. 
Our main result (Theorem \ref{itheo.unbounded}) is proved in section \ref{sec.ubnbdrk}: 
the unboundedness of the rank is first proved in \S\ref{sec.unbounded}, 
we then discuss the unboundedness of the average rank in \S\ref{sec.unbounded.redux}
and, finally, we make an observation in \S\ref{sec.amusing}   (Theorem \ref{theo.amusing})   providing a third proof of Theorem~\ref{itheo.unbounded} which is conditional to a certain hypothesis about primes.

\numberwithin{equation}{section}    


\section[The curves]{The curves $E_d$ and their invariants}
\label{sec.1}

\begin{center}
\emph{Throughout this article, we fix a finite field $\F_q$ of characteristic $p\geq 3$,
and we denote by $K=\F_q(t)$.}
\end{center}

\paragraph{}
For any integer $d\geq 1$, consider the elliptic curve $E_d/K$ defined by the affine Weierstrass model: 
 \begin{equation*}
E_d: \qquad 
Y^2+2t^d\cdot XY -4t^{2d}\cdot Y= X^3 -6t^d\cdot X^2+8t^{2d} \cdot X.	
\end{equation*}
The sequence $\{E_d\}_{d\geq 1}$ is an example of a 
\emph{Kummer family} of elliptic curves over $K$ (see \S\ref{sec.Ulmer}). 
This sequence corresponds to the special case `$a=1/2$' of Berger's example%
\footnote{Note that Berger assumes that $d$ is coprime to $p$, which we don't. 
We also point out a small typo in the Weierstrass model of $E_1$ given p.\,3029 of \cite{Berger}: in standard notations, the coefficient $a_2$ should be $-(a+1) t$ instead of $-(a+1)$.} 
 (6) in~\S4.3 of \cite{Berger}, further studied in~\S4.4 \emph{loc.\,cit.} as example (2).
The family $\{E_d\}_{{d\geq 1, (d,q)=1}}$ was also studied in \cite[Chap. 7]{Griffon_PHD}, where $E_d$ is denoted by $B_{1/2, d}$.
In section \ref{sec.construction} below, we will recall how these curves $E_d$ are constructed in~\cite{Berger}.  

\paragraph{}
From the above model of $E_d$, a straightforward computation shows that the $j$-invariant  of $E_d$ is
\begin{equation*}
j(E_d) =	\frac{2^4\cdot (t^{2d}+12)^3}{t^{2d}+16}.
\end{equation*}
We observe that $j(E_d)\in K$ is clearly not constant (so that $E_d/K$ is nonisotrivial) and not a $p$th power (\ie{} the extension $K/\F_q(j(E_d))$ is separable).
One can easily transform the model above into a short Weierstrass model (since $K$ has characteristic $\neq 2$) and obtain that $E_d$ is also given by  
\begin{equation}\label{eq.Wmodel} 
E_d:\qquad y^2=x\cdot\big(x^2+t^{2d}\cdot x-4t^{2d}\big).
\end{equation}
From now on, we work with this model of $E_d$, unless otherwise specified. 
The discriminant of \eqref{eq.Wmodel} 
is easily seen to be $\Delta_d = 2^8\cdot  t^{6d}(t^{2d} +16)$.

\begin{rema}\label{rema.kummer} 
We point out that our point of view on Kummer families is slightly different from that of Ulmer's and Berger's.
In their articles \cite{Berger, Ulmer_largeLrank, Ulmer_MWJacobians}, 
they fix an elliptic curve $E/\F_q(t)$ and consider the variation of the rank of $E(\F_{q}(t^{1/d}))$ as $d$ varies. 
In the present paper, the base field $\F_q(t)$ is fixed and, for all $d$, we study the ranks of $E_d(\F_q(t))$ as $d$ varies, where the elliptic curve $E_d/\F_q(t)$ is obtained from $E_1$ by `replacing each occurence of $t$ in a Weierstrass model of $E_1$ by $t^d$'. 
These two points of view are formally equivalent; we chose the latter because it is better suited to our purpose.
\end{rema}

\begin{rema}\label{rema.alld}
Let $d\geq 1$ be an integer. One can write $d=d' p^e$ for some integers $e\geq 0$ and $d'\geq 1$ such that $\gcd(d',p)=1$.
The elliptic curves $E_d$ and $E_{d'}$ are then isogenous over $K$. Indeed, the $p^e$th power Frobenius isogeny $F_{p^e}$ provides a (purely inseparable) $K$-isogeny $F_{p^e}: E_{d'}\to E_d$.

Being $K$-isogenous, the elliptic curves $E_{d'}$ and $E_d$ have the same conductor (see \cite[App. C]{Gross_bsd}, or \cite[Coro. VII.7.2]{AEC} and \cite[Ex. IV.40]{ATAEC}), 
the same $L$-function (see \cite[Chap. I, Lemma 7.1]{Milne_ArithDual_2nd}) 
and the same Mordell--Weil rank (because the kernel of an isogeny is finite).
\end{rema}

\subsection{Conductor of $E_d$}

Let us first describe the bad reduction of $E_d$ and deduce the degree of the conductor of $E_d$. 
We identify, as we may, the finite places of $K$ to monic irreducible polynomials in $\F_q[t]$.

For any integer $d\geq 1$ coprime to $q$, let $M_d$ be the set of (finite) places of $K$ that divide $t^{2d}+16$; 
the set $M_d$ can be also be viewed as the set of closed points of $\P^1_{/\F_q}$ corresponding to $2d$th roots of $-16$.

\begin{prop}\label{prop.badred}
For any $d\geq 1$ coprime to $q$, the elliptic curve $E_d$ given by \eqref{eq.Wmodel} has good reduction outside ${S=\{0\}\cup M_d\cup\{\infty\}}$, and 
\begin{itemize}
\item At $v=0$, $E_d$ has good (resp.  additive) reduction when $d$ is even (resp. $d$ odd),
\item At $v\in M_d$, $E_d$ has split multiplicative reduction,
\item At $v=\infty$, $E_d$ has split multiplicative reduction.
\end{itemize}
\end{prop}

\begin{proof}
This can be proved directly 
by applying Tate's algorithm to $E_d$ 
(see \cite[Chap. IV, \S9]{ATAEC}). This algorithm 
 actually yields    more   information about the bad reduction of $E_d$:
at $v=0$, $E_d$ has either good reduction (type $\mathbf{I}_0$) or additive reduction of type $\mathbf{I}_0^\ast$ depending on whether $d$ is even or odd, respectively. 
At $v\in M_d$ (resp. at $v=\infty$), $E_d$ has split multiplicative reduction of type $\mathbf{I}_1$ (resp. of type $\mathbf{I}_{4d}$).
\ProofEnd\end{proof}

\begin{coro}\label{coro.deg.cond}
Let $d\geq 1$ be an integer, and $E_d/K$ the elliptic curve defined by \eqref{eq.Wmodel}. 
We denote by $\cond_d$   the conductor divisor of $E_d$, and by $\cond^{f}_d$   the part of $\cond_d$ which is prime to $0$ and $\infty$.

Writing $d=d'\cdot p^e$ where $e\geq 0$ and $\gcd(d', p)=1$, one has 
\begin{equation}\label{eq.deg.cond}
\deg\cond_d =\deg\cond_{d'}= \begin{cases}
 2d'+1 & \text{if $d'$ is even}, \\
 2d'+3 & \text{if $d'$ is odd},
 \end{cases}
\quad\text{and} \quad
\deg\cond^{f}_d = \deg \cond^{f}_{d'} = 2d'.
\end{equation}
\end{coro}

\begin{proof} When $d$ is coprime to $p$, 
the result directly follows from the Proposition above and from the observation that one has  
$\sum_{v\in M_d} \deg v = \deg(t^{2d}+16) = 2d$ because $t^{2d}+16$ is squarefree in $\F_q[t]$.
The case of a general~$d$ reduces to this first case: indeed, $E_d$ and $E_{d'}$ have the same conductor since they are isogenous \emph{via} the $p^e$th Frobenius isogeny (see Remark \ref{rema.alld}).
\ProofEnd\end{proof}

\begin{rema}\label{rem.minimod} 
The application of Tate's algorithm at each place $v$ of $K$ also provides us with a minimal $v$-integral model of $E_d$. 
  From the proof of the above Proposition and the expression of the discriminant of~\eqref{eq.Wmodel}, one already deduces that the model \eqref{eq.Wmodel} is integral and minimal at all places $v\neq 0, \infty$ of $K$.

At $v=0$, the model \eqref{eq.Wmodel} is integral but is never minimal: Tate's algorithm yields that 
\begin{itemize}
\item When $d$ is even, a minimal model of $E_d$ at $v=0$ is given by $y^2=x(x^2+t^d\cdot x-4)$.
\item When $d$ is odd, a minimal model of $E_d$ at $v=0$ is given by $y^2=x(x^2+t^{d+1}\cdot x - 4 t^2)$.
\end{itemize}
\end{rema}

\subsection{Some rational points on $E_d$}
When one looks for nontrivial $K$-rational points on $E_d$, one quickly finds at least two: $P_0=(0,0)$ and $P_d=(2t^d, 2t^{2d})$ on the model \eqref{eq.Wmodel}.
In this subsection, we show that $P_0$ is a $2$-torsion point (and is the only nontrivial such one) and that $P_d$ has infinite order.

\begin{theo} \label{theo.tors}
Let $d\geq 1$ be an integer. 
Then 
$ E_d(K)\tors \simeq \Z/2\Z$.
More precisely, $E_d(K)\tors$ is generated by $P_0=(0,0)$.
\end{theo}

\begin{proof} 
For the duration of the proof, we denote the torsion subgroup of $E_d(K)$ by $G$. 
Since the $j$-invariant of $E_d$ is not a~$p$th power in $K$, 
$E_d(K)$ has no torsion points of order a power of~$p$  by Proposition~7.1 in \cite[Lect. 1]{UlmerParkCity}; in other words, the $p$-primary part $G[p^\infty]$ of $G$ is trivial.
We can also describe the $2$-torsion in $G$ explicitly.
In the homogenised version of the model \eqref{eq.Wmodel}, the four $\bar{K}$-rational $2$-torsion points on $E_d$ are given by 
$[0:1:0]=\O$, $[0:0:1]=P_0$, 
$[-t^d(t^{d}+\sqrt{t^d+16})/{2}:0:1]$ and $[-t^d(t^{d}-\sqrt{t^d+16})/{2}:0:1]$. 
Clearly, only the first two are $K$-rational (since $t^d+16\in\F_q[t]$ is squarefree), hence $G[2] = \{\O, P_0\}$.

For any place $v$ of $K$, we let $\Gamma_v$ denote the group of components of the special fibre at $v$ of the N\'eron model of $E_d$ (so that $\Gamma_v=\{1\}$ for all but finitely many places). 
Our proof of Proposition \ref{prop.badred} and the table in \cite[\S 7.2]{SchShio} provide explicit descriptions of the groups $\Gamma_v$ for $E_d$.

\medskip

When $d$ is odd, the curve $E_d$ has additive reduction at $\infty$: Lemma 7.8 in \cite{SchShio} shows that the prime-to-$p$ part of $G$ embeds into $\Gamma_\infty\simeq(\Z/2\Z)^2$. 
We thus infer that the whole of $G$ is $2$-torsion, so that $G=G[2]=\{\O, P_0\}$ by the computation in the previous paragraph. 
This complete the proof in this case.

 \medskip

We now turn to the case when $d$ is even. 
Corollary 7.5 in \cite{SchShio} shows that $G$ embeds as a subgroup of $\Gamma:=\prod_{v} \Gamma_v$. 
More explicitly, this result yields that $G$ is isomorphic to a subgroup of $\Z/4d\Z$. 
In particular, $G$ has to be cyclic, and we let $g$ denote its order. 
Upon carrying out an explicit version of the proof of Proposition 7.1 in \cite[Lect. 1]{UlmerParkCity}, one can show that $1\leq g\leq 12$ (see the proof of \cite[Prop. 1.5]{Griffon_Hessian}).
Given the previous constraints on $g$, we obtain that $g\in\{2, 4, 6, 8, 10, 12\}$: hence we will be done once we have checked that $E_d(K)$ contains no nontrivial $3$-torsion,  $4$-torsion or $5$-torsion. 

We carry out this last step by considering \emph{division polynomials}: for any $n\geq 1$, there is a rational map $\psi_n$ on $E_d$ (which can be explicitly given  in terms of Weierstrass coefficients of $E_d$) with the following property. 
The zeroes of $\psi_n$ are simple and they are exactly the nontrivial $n$-torsion points on $E_d$, and $\psi_n$ has a pole of the appropriate order at the neutral element $\O$. 
The reader is referred to \cite[Chap. III, Ex. 3.7]{AEC} for more details.

First we note that, when~$d$ is even, the Weierstrass equation $W' : y^2=x(x^2+t^dx-4)$ is a minimal integral model of $E_d$ at  all finite places of $K$ (by `scaling' \eqref{eq.Wmodel} \emph{via} $(x,y)\mapsto (t^dx, t^{3d/2}y)$). 
The formulae in \cite[Chap. III, Ex. 3.7]{AEC} here read: 
\begin{align*}
\psi_3 
&= 3x^4+4t^{d}\cdot x^3-24\cdot x^2-16, \\
\psi_4 
&= 4y\big(x^6 +2t^d\cdot x^5 -20\cdot x^4 -80\cdot x^2 -32t^d\cdot x+64\big), \\
\psi_5&= 
5x^{12} +20t^{d}\cdot x^{11} +8t^{d}(2t^{d}-31) \cdot x^{10} -320t^{d}\cdot x^9 -1680\cdot x^8 -5760t^{d}\cdot x^7 -3840(t^{2d}-5)\cdot x^6	\\
&\qquad\qquad\qquad
-1024t^{d}(t^{2d}-23)\cdot x^5 +1280(t^{2d}-3125)\cdot x^4 -35840t^{d}\cdot x^3 +51200\cdot x^2 +4096.
\end{align*}
Therefore, to conclude our proof, it remains to show that, for all $P=(x_P, y_P)\in E_d(K)\smallsetminus\{\O\}$ 
one has $\psi_n(x_P, y_P)\neq 0$ for $n\in\{3, 4, 5\}$.

Assume that there exists a nontrivial $3$-torsion point $P=(x_P, y_P)\in E_d({K})$. 
By construction of the polynomial $\psi_3$, we have $\psi_3(P)=0$ \ie{}, we have
\[x_P^4 +\tfrac{4}{3}t^{d}\cdot x_P^3 - 8\cdot x_P^2 - \tfrac{16}{3}=0.\]
In particular, $x_P\in{\F_q(t)}$ is integral over $\F_q[t]$ and we obtain that $x_P\in\F_q[t]$. 
Now, the equation above can be rewritten as:
$ x_P^2\cdot \left(3x_P^2+4t^{d}\cdot x_P - 24\right) = 16$.
For degree reasons, the polynomials $x_P$ and $x_P^2+4t^{d}x_P-24$ must both be constant and nonzero, which is impossible. 
This contradiction shows that $E_d(K)$ has no nontrivial $3$-torsion.
The same argument also excludes $5$-torsion: 
if there were a nontrivial $5$-torsion point $P=(x_P, y_P)\in E_d(K)$, then $x_P$ would be in $\F_q[t]$ and would satisfy 
\[ \big(5x_P^{10}
+20t^{d}\cdot x_P^{9}
+8t^{d}(2t^{d}-31) \cdot x_P^8
+\dots 
-35840t^{d}\cdot x_P
+51200\big)\cdot x_P^2=-4096.\]
Such an equality is never satisfied, so that $E_d(K)[5]=\{\O\}$.
Finally, assume that there is a point $P=(x_P, y_P)\in E_d({K})$ of exact order $4$. 
By construction $\psi_4(P)=0$ but, since $P$ is not $2$-torsion, one has $y_P\neq 0$. 
By the same argument as above, $x_P$ is an element of $\F_q[t]$ and is a solution of 
\[\big(x_P^5 +2t^d\cdot x_P^4 -20\cdot x_P^3 -80\cdot x_P -32t^d \big)\cdot x_P = -64.\]
Since no element of $\F_q[t]$ can satisfy such a relation, $P$ cannot exist and we have shown that $G[4]=\{\O\}$.
This concludes the proof that $G=\{\O, P_0\}$ in the case when $d$ is even. 
\ProofEnd\end{proof}

\begin{coro}\label{coro.positive.rk}
For any integer $d\geq1$, the point $P_d=(2t^d, 2t^{2d})\in E_d(K)$ has infinite order.
In particular, the elliptic curve $E_d$ has positive rank. 
\end{coro}

\begin{proof} 
Given the previous Theorem, it suffices to check that $P_d\neq P_0$, which is obvious. 
\ProofEnd\end{proof}

 \subsection{Ulmer's theorem on unbounded ranks}
 \label{sec.Ulmer}

Let $\F_q$ be a finite field of characteristic $p\geq 3$ and $K=\F_q(t)=\F_q(\P^1)$. 
For any elliptic curve $E/K$, we denote by $\cond^f_{E/K}$ be the part of the conductor $\cond_{E/K}\in\mathrm{Div}(\P^1)$ that is prime to $0$ and $\infty$. 

Given a nonisotrivial elliptic curve $E'$ over $K$, one can construct as follows a sequence $\{E'_d\}_{d\geq 1}$ of elliptic curves over $K$, 
which we call the \emph{Kummer family built from $E'$}. 
For any integer $d\geq 1$, let $E'_d/K$ be the pullback of $E'/K$ under the Kummer map $\P^1\to\P^1$ given by $t\mapsto t^d$. 
In down-to-earth words: given a Weierstrass model $\mathcal{W}$ of $E'/K$, define $E'_d$ by the Weiestrass model  obtained by replacing each occurence of $t$ in $\mathcal{W}$ by $t^d$.

 \begin{theo}[Ulmer]\label{theo.Ulmer}
 Let $E'/K$ be a nonisotrivial elliptic curve: 
we assume that the degree of $\cond^f_{E'/K}$ is \emph{odd} and we denote by $\{E'_d\}_{d\geq 1}$ be the Kummer family built from $E'$. 
Then the sequence of analytic ranks $(\ord_{T=q^{-1}} L(E'_d/K, T))_{d\geq 1}$ is unbounded. 
Moreover,  for $d_n:=q^n+1$ where $n\geq 1$ is an integer, one has
 \[\ord_{T=q^{-1}}L(E'_{d_n}/K, T) 
 \geq \frac{q^n+1}{2n} - C, \]
 where the constant $C$ depends at most on $q$ and $E'$.
  \end{theo}
 
The sequence $\{E_d\}_{d\geq 1}$ of elliptic curves $E_d$ defined by \eqref{eq.Wmodel} is an example of a Kummer family: it is the one built from $E_1 : y^2=x(x^2+t^2x-4t^2)$.
As was noted in \cite[\S4.4]{Berger}, the Theorem above \emph{does not} apply to the Kummer family at hand, because $\deg\cond^f_1=\deg\cond^f_{E_1/K}$ is even (see Corollary \ref{coro.deg.cond}).


\section{Characters of order dividing $2d$ and some character sums}
\label{sec.prelim}

The goal of the next section will be to give an explicit expression for the $L$-function of $E_d$: the result will be expressed  in terms of certain character sums, which we start by introducing and relating to products of some Jacobi sums.

\subsection{Jacobi sums} 
\label{sec.jacobi}

Let $\F$ be a finite field of odd characteristic. 
There is a unique nontrivial character of order $2$ on $\F^\times$, which we denote by%
\footnote{or simply by $\lambda$ when no confusion is likely.}
 $\lambda_\F:\F^\times\to\{\pm1\}$.
We follow the convention that nontrivial multiplicative characters $\chi:\F^\times\to\Qbar^\times$ are extended to the whole of $\F$ by $\chi(0):=0$, 
and the trivial character $\trivcar$ by $\trivcar(0):=1$.
To any pair of multiplicative characters $\chi_1, \chi_2$ on $\F^\times$, we associate a Jacobi sum
\[\ja_{\F}(\chi_1, \chi_2) := -\sum_{\substack{ x_1, x_2 \in\F \\ x_1+ x_2=1}} \chi_1(x_1)\chi_2(x_2).\]
We refer to \cite[Chap. II, \S5.3 - \S5.4]{Cohen} for more details about these sums. For the convenience of the reader, we recall the following facts: 
\newcommand{\refJac}[1]{\hyperref[#1]{\textsf{(Jac\,\ref*{#1})}}}
\begin{enumerate}[\sf({Jac}\,1)]
\item\label{eq.jac.3} 
	For any nontrivial character $\chi$, 
	one has $	\ja_\F(\trivcar, \chi)=0$.
\item\label{eq.jac.2}
	For any nontrivial character $\chi$ such that $\chi\neq\lambda$, 
	 one has $\ja_\F(\chi, \lambda)=\chi(4)\cdot\ja_\F(\chi, \chi)$.
\item\label{eq.jac.1}
	One has $	\ja_\F(\lambda, \lambda)=\lambda(-1)$.
\item\label{eq.HD.jac}	
	For any characters $\chi_1, \chi_2$ on $\F^\times$ and any finite extension $\F'/\F$, one has 
	\[\ja_{\F'} \big(\chi_1\circ\norm_{\F'/\F}, \chi_2\circ\norm_{\F'/\F} \big) =  \ja_\F(\chi_1, \chi_2) ^{[\F':\F]},\]
	where $\norm_{\F'/\F}:\F'\to\F$ denotes the relative norm of the extension $\F'/\F$.
	This identity is often called the Hasse--Davenport relation for Jacobi sums (see \cite[Coro III.5.7]{Cohen}).
\item\label{eq.jac.absval}
	 For any characters $\chi_1, \chi_2$ on $\F^\times$  such that none of $\chi_1$, $\chi_2$,  and $\chi_1\chi_2$ is trivial, one has 
	 \[|\ja_{\F}(\chi_1, \chi_2)| = |\F|^{1/2},\] 
	 in any complex embedding of $\Qbar$.
\end{enumerate}
 
 \subsection{Two identities with character sums}
\label{sec.identity}

In this subsection, $\F$ denotes a finite field of odd characteristic. 
For any multiplicative character $\chi:\F^\times\to\Qbar^\times$, we consider the double character sum
\begin{equation}\label{eq.doublecharsum}
B(\F,\chi) := \sum_{x\in\F}\sum_{z\in\F^\times}\chi(z)\cdot \lambda\left(x^3+x^2z-4xz\right).	
\end{equation}

Let us first prove the following:
\begin{prop}\label{prop.charsum.id}
Let $\chi:\F^\times\to\Qbar^\times$ be a multiplicative character. Then
\begin{equation}\notag{}
B(\F,\chi) = \begin{cases}
|\F| & \text{ if } \chi=\trivcar \text{ is trivial},\\
1 & \text{ if } \chi=\lambda \text{ has order }2,\\
\ja_\F(\chi, \chi) & \text{ if } \chi  \text{ has order $4$},\\
\chi^2(4)\cdot \ja_\F(\chi,\chi)\cdot\ja_\F(\lambda\chi^2, \bar{\chi})& \text{ otherwise.}
\end{cases}
\end{equation}
\end{prop}
Note that characters $\chi:\F^\times\to\Qbar^\times$ of exact order $4$ exist if and only if $|\F|\equiv 1\bmod{4}$ (by cyclicity of the group $\F^\times$). 
If this is the case, there are exactly two such characters. 

\begin{proof} 
Notice that the terms with $x=0$ do not contribute to the sum \eqref{eq.doublecharsum} because $\lambda(0)=0$, therefore
\[B(\F,\chi) = \sum_{x\in\F^\times} \lambda(x)\left(\sum_{z\in\F^\times}\chi(z)\cdot \lambda\left((x-4)z+x^2\right)\right).\]
For a given $x\neq 0$, consider the inner sum in the above equality: if $x=4$, then
\[\sum_{z\in\F^\times}\chi(z)\cdot \lambda\left((x-4)z+x^2\right) 
= \lambda(4)^2\cdot \sum_{z\neq 0} \chi(z) = \begin{cases}
 |\F|-1 & \text{ if } \chi=\trivcar,	 \\
0 & \text{ otherwise.}
 \end{cases} \]	
 If $x\neq 4$, we let $A_x:=x^2/(x-4)$ and we obtain that
 \[\sum_{z\in\F^\times}\chi(z)\cdot \lambda\left((x-4)z+x^2\right)  = \lambda(x-4)\cdot \sum_{z\neq 0} \chi(z)\lambda(z+A_x).\]
If $\chi=\trivcar$, it is not hard to check that $\sum_{z\neq 0} \chi(z)\lambda(z+A_x) = -1$. 
If $\chi$ is nontrivial, we can add the term $z=0$ (which does not contribute) and reindex the sum by letting $(u_1, u_2) = (-z/A_x, 1+z/A_x)$. 
This leads~to
\[\sum_{z\neq 0} \chi(z)\lambda(z+A_x) = \chi(-A_x)\lambda(A_x)\cdot\sum_{\substack{u_i\in\F \\ u_1+u_2=1}} \chi(u_1)\lambda(u_2)
= - \chi(-1)\lambda(A_x)\chi\left(A_x\right)
\cdot\ja_{\F}(\chi, \lambda). \]
Besides, one can show that $\sum_{x}\lambda(x(x-4))=-1$ (see \cite[Theorem 5.48]{LidlN}). 
If $\chi$ is trivial, this identity and what we have just proved yield
\[B(\F,\trivcar) = -\sum_{x\neq 0, 4} \lambda(x)\lambda(x-4) + \lambda(4)\cdot(|\F|-1) 
= 1 + |\F|-1 = |\F|,\]
as was to be shown. 
We now assume that $\chi\neq\trivcar$: using the expression of $A_x$ and the fact that $\lambda(x^2)=\lambda(x-4)^2=1$ for all $x\neq 0, 4$,  we have obtained thus far that 
\begin{equation}\label{eq.char.id.inter1}
B(\F,\chi) = -\chi(-1)\cdot\ja_\F(\chi, \lambda)\cdot B'(\F,\chi), \quad \text{ with }B'(\F,\chi):=\sum_{x\neq 0, 4} \lambda(x)\chi(x^2)\bar{\chi}(x-4).
\end{equation}
It remains to `compute' $B'(\F,\chi)$, and we distinguish several cases. 
We first assume that $\chi=\lambda$ has order exactly $2$: in this case, one has $\lambda(x)\chi(x^2)\bar{\chi}(x-4)=\lambda(x)\lambda(x-4)$ for all $x\in\F\smallsetminus\{0, 4\}$. 
Hence, $B'(\F, \chi) = \sum_{x\in\F}\lambda\big(x(x-4)\big)=-1$ by the same identity as before.
Moreover, \refJac{eq.jac.1} yields that $\ja_\F(\lambda, \lambda)=\lambda(-1)$ 
so that $B(\F,\chi)=\lambda(-1)^2=1$. 
Next, we assume that $\chi\notin\{\trivcar, \lambda\}$ and that $\lambda\chi^2$ is the trivial character (this happens exactly when $\chi$ has order $4$): 
for all $x\in\F\smallsetminus\{0,4\}$, we now have $\lambda(x)\chi(x^2)\bar{\chi}(x-4)=\bar{\chi}(x-4)$. 
Hence $B'(\F,\chi)=-\bar{\chi}(-4)$ and we have $B(\F,\chi)=\bar{\chi}(4)\cdot \ja_\F(\chi, \lambda)$ which, by \refJac{eq.jac.2}, can be rewritten as $B(\F,\chi)= \ja_\F(\chi, \chi)$.

Finally, we deal with the case where $\chi^4\neq \trivcar$. 
Then $\lambda\chi^2$ is nontrivial and, by setting $x=4x'$, we obtain that
\[B'(\F,\chi)  = \sum_{x\in\F} (\lambda\chi^2)(x)\bar{\chi}(x-4)
= \lambda(4)\cdot\chi(-4)\cdot \sum_{x'\in\F} (\lambda\chi^2)(x')\bar{\chi}(1-x')
= -\chi(-4)\cdot \ja_\F(\lambda\chi^2, \bar{\chi}). \]
Appealing to \refJac{eq.jac.2} once more, we conclude that 
\[B(\F,\chi) = \chi(-1)^2\chi(4)\cdot\ja_\F(\chi, \lambda)\cdot\ja_\F(\lambda\chi^2, \bar{\chi}) 
= \chi(16)\cdot\ja_\F(\chi, \chi)\cdot\ja_\F(\lambda\chi^2, \bar{\chi}). \eqno \Box\]
\end{proof}
 
 From \refJac{eq.HD.jac} and the above Proposition, one directly obtains the following:
 \begin{coro}\label{coro.HD.B}
 Let $\chi:\F^\times\to\Qbar^\times$ be a multiplicative character. For any finite extension $\F'/\F$, one has
 \[B(\F', \chi\circ\norm_{\F'/\F}) =  B(\F, \chi)^{[\F':\F]},\]
 where $\norm_{\F'/\F}:\F'\to\F$ denotes the relative norm of the extension.
\end{coro}

We end this subsection by showing that:

\begin{prop} \label{prop.charsum.id2}
Let $\F$ be a finite field of odd characteristic. 
Then one has
\begin{equation}\label{eq.char.id2}
\sum_{x\in\F}	\lambda\big(x(x^2-4))\big) = - \sum_{\substack{\chi^4=\trivcar \\ \chi^2\neq \trivcar}}   \ja_\F(\chi, \chi),
\end{equation}
where the sum on the right-hand side is over all multiplicative characters of exact order $4$ (hence it  vanishes 
if $|\F|\not\equiv 1\bmod{4}$).
\end{prop}
\begin{proof} 
We denote by $S_\F$ the sum on the left-hand side of \eqref{eq.char.id2}. 
We note that $\lambda(-1) = (-1)^{(|\F|-1)/2}$, so that $\lambda(-1) =1$ if and only if $|\F|\equiv1\bmod{4}$.
This happens if and only if there exist characters $\chi:\F^\times\to\Qbar^\times$ of order $4$. 
Reindexing the sum $S_\F$ by setting $x'=-x$, we obtain that $S_\F = \lambda(-1)\cdot S_\F$. 
Hence, $S_\F=0$ when $\lambda(-1)=-1$; and \eqref{eq.char.id2} is proved in this case, since the right-hand side is then an empty sum. 

We now assume that $|\F|\equiv 1\bmod{4}$ and denote by $\theta:\F^\times\to\Qbar^\times$ one of the two characters of order $4$ (the other one being $\lambda\theta$). 
Since $\lambda = \theta^2$, the sum $S_\F$ can be rewritten as
\[S_\F = \sum_{x\in\F}\lambda(x)\lambda(x^2-4) = \sum_{x\in\F}\theta(x^2)\lambda(x^2-4) = \sum_{z\in\F}(1+\lambda(z))\cdot\theta(z)\lambda(z-4)
= \sum_{ \chi\in\{\theta, \lambda\theta\}} \left(\sum_{z\in\F} \chi(z)\lambda(z-4)\right).\]
Indeed, for any $z\in\F$, the number of $x\in\F$ such that $x^2=z$ equals $1+\lambda(z)$.
For a nontrivial character $\chi:\F^\times\to\Qbar^\times$, setting $(u_1, u_2)=(z/4, 1-z/4)$ allows to identify the following sum to a Jacobi sum:
\[\sum_{z\in\F} \chi(z)\lambda(z-4) = {\chi}(4){\lambda}(-4) \cdot \sum_{\substack{u_i\in\F \\ u_1+u_2=1}} \chi(u_1)\lambda(u_2)
= - {\chi}(4){\lambda}(-4)\cdot\ja_\F(\chi, \lambda).\]
Hence, using \refJac{eq.jac.2} and noting that $\lambda(-4)=1$ and $\theta^2(4)=\lambda(4)=1$, the above equality applied successively to $\chi=\theta$ and $\chi=\lambda\theta$ implies that 
\[-S_\F 
= {\theta}(4)\lambda(-4)\cdot \ja_\F(\theta, \lambda) + \theta(4)\lambda(4)\lambda(-4)\cdot \ja_\F(\lambda\theta, \lambda)
=  \ja_\F(\theta, \theta) +   \ja_\F(\lambda\theta, \lambda\theta).  \eqno\square\]
\end{proof}

\subsection{Action of $q$ on $\Z/2d\Z$}
\label{sec.q.act}

For any integer $D\geq 1$ coprime to $q$, there is a natural action of the subgroup $\langle q\rangle_d\subset(\Z/D\Z)^\times$ generated by $q$ on $\Z/D\Z$ by multiplication%
\footnote{We will simply say that `$q$ acts on $\Z/D\Z$ by multiplication'.}. 
For any subset $Z\subset \Z/D\Z$ which is stable under this action, we denote by $\O_q(Z)$ the set of orbits of $Z$. 
Given such a $Z\subset\Z/D\Z$ and an orbit $\Mm\in\O_q(Z)$, we will often need to make a choice of representative $m\in Z$ of this orbit: we make the convention that orbits in $\O_q(Z)$ are always denoted by a bold letter ($\Mm$, $\mathbf{n}$,~...) and that the corresponding normal letter ($m$, $n$,~...) designates any choice of representative in $Z$ of this orbit.

For any integer $n\geq 1$ coprime to $q$, we denote the (multiplicative) order of $q$ modulo $n$ by $o_q(n)$ \ie{}, the order of the subgroup $\langle q\rangle_n\subset(\Z/n\Z)^\times$ generated by $q$.
For any orbit $\Mn\in\O_q(Z)$, its length $|\Mn| = \left| \left\{ n, qn, q^2n, \dots\right\}\right|$ is equal to
\[|\Mn| = \min\left\{ \nu\in\Z_{\geq 1} \ \big| \ q^\nu n \equiv n \bmod{D}\right\},\]
which, in turn, equals 
$|\Mn|= o_q(D/\gcd(D, n))$ 
for any~$n\in\Mn$.
For any power $q^v$ of $q$, note that $q^v n\equiv n \bmod{D}$ if and only if $|\Mn|$ divides $v$, \ie{} if and only if $\F_{q^v}$ is an extension of $\F_{q^{|\Mn|}}$ (by construction of the order).

For any integer $d\geq 2$ which is coprime to $q$, we will be particularly interested in the subset  
\[Z_{2d} := \begin{cases}
 \Z/2d\Z\smallsetminus\{0, d/2, d, 3d/2\} &\text{ if $d$ is even},\\
 \Z/2d\Z\smallsetminus\{0, d\} &\text{ if $d$ is odd}
 \end{cases}
\]
of $\Z/2d\Z$ (which is stable under multiplication by $q$ because $\gcd(2d,q)= 1$) and in the corresponding set of orbits $\O_q(Z_{2d})$. 

\begin{rema}\label{rema.trivq.act}
 In the special case when $2d$ divides $q-1$ (\ie{} when $q\equiv 1\bmod{2d}$), the action of $q$ on $Z_{2d}$ is trivial and there is a bijection between $\O_q(Z_{2d})$ and $Z_{2d}$.
\end{rema}

\subsection{Characters of order dividing $2d$}
\label{sec.char}

Let us fix an algebraic closure $\bar{\Q}$ of $\Q$ and a prime ideal~$\gP$ above $p$ in the ring of algebraic integers $\bar{\Z}\subset \bar{\Q}$.
The residue field $\bar{\Z}/\gP$ can be viewed as an algebraic closure $\bar{\F_p}$ of $\F_p$. 
Moreover, the reduction map $\bar{\Z}\to\bar{\Z}/\gP$ induces an isomorphism between the group $\mu_{\infty, p'}\subset\bar{\Z}^\times$ of roots of unity of order prime to $p$ and the multiplicative group $\bar{\F_p}^\times$. 
We let $\tt:\bar{\F_p}^\times\to\mu_{\infty, p'}$ be the inverse of this isomorphism, and we denote by the same letter the restriction of $\tt$ to any finite field $\F_q$ (which we view as as a subfield of $\bar{\Z}/\gP$). 
Note that any nontrivial multiplicative character on any finite extension of $\F_q$ is then a power of $\tt$.

Suppose we are given a prime-to-$p$ integer $d\geq 1$; for any $n \in \Z/{2d}\Z\smallsetminus\{0\}$, let $\Mn\in \O_q(\Z/2d\Z)$ be the corresponding orbit and  define a multiplicative character $\tt_n : \F_{q^{|\Mn|}}^\times\to \Qbar^\times$ by
\[\forall x\in\F_{q^{|\Mn|}}^\times, \quad  \tt_n(x) = \tt(x)^{(q^{|\Mn|}-1)n/2d} \quad \text{ and put } \tt_n(0):=0. 
\]
Further, if $s\geq 1$, we  `lift' $\tt_n$ to a character $\tt_n^{(s)}:\F_{q^{s|\Mn|}}^\times\to\Qbar^\times$ via the relative norm $\F_{q^{s|\Mn|}}\to\F_{q^{|\Mn|}}$, \ie{} we let $\tt_n^{(s)}:=\tt_n\circ \norm_{\F_{q^{s|\Mn|}}/\F_{q^{|\Mn|}}}$.

The order of $\tt_n$ obviously divides $2d$ and a more careful computation shows that $\tt_n$ has exact order  $2d/\gcd(2d, n)$. 
Note that the order of $\tt_n$ does not divide $4$ when $n\in Z_{2d}\subset\Z/2d\Z$.
Since the norm is surjective, $\tt_n^{(s)}$ has the same order as $\tt_n$ for all $s\geq 1$.

The interest of the characters $\tt_n^{(s)}$ (with $n\in\Z/2d\Z\smallsetminus\{0\}$, $s\geq 1$) lies in the following result:
\begin{lemm}\label{lemm.reindex}
Let $d\geq 2$ be coprime to $q$, and $\F$ be a finite extension of $\F_q$. 
Denote by $X_{2d}(\F)$ the set of nontrivial characters $\chi$ on $\F^\times$ such that $\chi^{2d}=\trivcar$ and $\chi^4\neq \trivcar$. Then
\[ X_{2d}(\F) 
= \left\{ \tt_n^{(s)}, \ n\in Z_{2d} \text{ and } s\geq 1 \text{ such that } s\cdot |\Mn| = [\F:\F_q]\right\}.\]
\end{lemm}
In particular, as $n$ runs through $Z_{2d}$ and $s$ runs through the positive integers, $\tt_n^{(s)}$ runs through all characters  on all finite extensions of $\F_q$ whose order divides $2d$ and whose 
$4$th power is nontrivial. 
 
\begin{proof}
The detailed proof of a very similar result can be found in \cite[\S2.2]{Griffon_Hessian}. 
Let us write $\F=\F_{q^n}$, denote by $C$ the (cyclic) group $C$ of multiplicative characters on $\F_{q^n}^\times$ 
whose order divides $2d$ and 
 ${e_n:=\gcd(2d, q^n -1)}=|C|$.
The main point is to note that the character $\cchi:=\tt^{(q^n-1)/e_n}$ on $\F_{q^n}^\times$ has  exact order $e_n$. Therefore $\cchi$ generates $C$: any element of $C$ is of the form $\cchi^k$ for $k\in\{0, \dots, 2d-1\}$. 
To obtain the enumeration of $X_{2d}(\F_{q^n})\subset C$, it only remains to exclude  the characters in $C$ whose order divide $4$: these correspond to $\cchi^k$ with $k\in\{0, d/2, d, 3d/2\}$. 
\ProofEnd\end{proof}

\begin{rema}\label{rema.trivq.act2} 
In the special case when $2d$ divides $q-1$ (\ie{} when $q\equiv 1\bmod{2d}$), all the characters $\tt_m$ (with $m\in Z_d$) are characters on $\F_q^\times$ because $|\Mm|=1$ (see Remark \ref{rema.trivq.act}). 
Moreover we remark that $\tt_m = (\tt_1)^m$ for all $m\in Z_d$, and that $\tt_1$ has exact order $2d$. 
Hence, by the Lemma above, as $m$ runs through $Z_{2d}$, the characters $(\tt_1)^m$ enumerate all possible nontrivial characters $\chi$ on $\F_q^\times$ such that $\chi^{2d}=\trivcar$ and $\chi^4\neq\trivcar$.
\end{rema}

 \subsection[Character sums]{The sums $\bbeta(n)$}
 \label{sec.bbeta}
 
We can finally introduce the following notation: 
\begin{defi}\label{defi.jacobi} Let $d\geq 1$ be coprime to $q$. For any $n\in Z_{2d}$, we let $Q:=q^{|\Mn|}$ and set 
\begin{equation}\label{eq.defi.jacobi}
\bbeta(n) :=  
\tt_n^{2}(4) \cdot \ja_{\F_Q}(\tt_n, \tt_n)\cdot \ja_{\F_Q}(\lambda_{\F_Q}\cdot\tt_n^2, \tt_n^{-1}).
\end{equation}
 For any orbit $\Mn\in\O_q(Z_{2d})$, we let  $\bbeta(\Mn):=\bbeta(n)$ for any choice of $n\in\Mn$. 
\end{defi}
 
We compile a few results about these numbers $\bbeta(n)$:

\begin{prop}\label{prop.beta}
Let $d\geq 1$ be coprime to $q$. 
The following statements hold:
\begin{enumerate}[(i)]
\item\label{item.beta.defined}
For any $\Mn\in\O_q(Z_{2d})$, $\bbeta(\Mn)$ is well-defined (\ie{} $\bbeta(q\cdot n) = \bbeta(n)$ for all $n\in Z_{2d}$),
\item\label{item.beta.B}
 In the notations of \S\ref{sec.identity},  $\bbeta(n)= B(\F_{q^{|\Mn|}},\tt_n)$ for all $n\in Z_{2d}$, 
\item\label{item.beta.HD}
 For all $n\in Z_{2d}$ and for any $s\geq 1$, we have $B\big(\F_{q^{s\cdot|\Mn|}},  \tt_n^{(s)}\big)=\bbeta(n)^{s}$,
 \item\label{item.beta.magnitude}
 One has $|\bbeta(n)|=q^{|\Mn|}$ for all $n\in Z_{2d}$, 
\end{enumerate}
\end{prop}
 
\begin{proof}
Let $n\in Z_d$ and set $Q:=q^{|\Mn|}$. 
The map $x\mapsto x^q$ is a bijection of $\F_{Q}$ and $\lambda_{\F_Q}= \lambda_{\F_Q}^q$, 
thus $\ja_{\F_Q}(\tt_n, \tt_n)=\ja_{\F_Q}(\tt_{q\cdot n}, \tt_{q\cdot n})$ and $\ja_{\F_Q}(\lambda_{\F_Q}\tt_n^2, \tt_n^{-1})=\ja_{\F_Q}(\lambda_{\F_Q}\tt_{q\cdot n}^2, \tt_{q\cdot n}^{-1})$. 
Moreover, $\tt_n(4) = \tt_n(4)^q = \tt_{q\cdot n}(4)$ since $4\in\F_q^\times$,
hence $\bbeta(n)=\bbeta(q\cdot n)$; 
a repeated application of this identity implies that the value of $\bbeta(n)$ is constant along the orbit $\Mn$, which shows that \eqref{item.beta.defined} holds. 
Item \eqref{item.beta.B} is a direct consequence of Proposition \ref{prop.charsum.id}
upon remarking that the order of $\tt_n$ does not divide $4$. 
Again by construction of $Z_{2d}$, none  of $\tt_n$, $\tt_n^2$ and $\tt_n^4$ is trivial, so that $|\bbeta(n)|=q^{|\Mn|}$ by combining Proposition \ref{prop.charsum.id} and \refJac{eq.jac.absval}; thus \eqref{item.beta.magnitude} is proved.
Finally, we deduce from Corollary~\ref{coro.HD.B} that 
$B(\F_{Q^s}, \tt_n^{(s)}\big) = B(\F_Q, \tt_n)^s$ for all $s\geq 1$; item \eqref{item.beta.B} above then yields \eqref{item.beta.HD}.
 \ProofEnd\end{proof}


\section[The L-function]{The $L$-function}
\label{sec.Lfunc}   

In this section, we give an explicit expression for the $L$-function of $E_d/K$. 
Before we do so, let 
us first recall the definition of $L(E_d/K, T)$. 

Let $d\geq 1$ be an integer. For any place $v$ of $K$, let $q_v$ be the cardinality of the residue field $\F_{v}$ of $K$ at $v$. 
For such a $v$, we denote by $(\widetilde{E_d})_v$ the plane cubic curve over $\F_v$ obtained by reducing modulo $v$  a minimal $v$-integral model of $E_d$, and we put $a_v := |\F_v| +1 - |(\widetilde{E_d})_v(\F_{v})|$.
If $v$ is a place of bad reduction for $E_d$, notice that $a_v$ is $0, +1$ or~${-1}$ depending on whether the reduction of $E_d$ at $v$ is additive, split multiplicative or nonsplit multiplicative, respectively. 
Recall that the $L$-function of $E_d/K$ is  the power series given by
\begin{equation}\label{eq.def.Lfunc}
 L(E_d/K, T) = \prod_{v\text{ good }} \left(1 - a_v \cdot T^{\deg v} + q_v \cdot  T^{2\deg v}\right)^{-1} \cdot
\prod_{v\text{ bad }} \left(1 - a_v \cdot  T^{\deg v}  \right)^{-1},
\end{equation}
where the products are over places of $K$ of good (resp. bad) reduction for $E_d$.
The reader may consult \cite[Lect.\,1, \S9]{UlmerParkCity} or \cite[Lect. 2, \S2]{Gross_bsd} for more details. 

The curve $E_d/K$ being nonisotrivial, a theorem of Grothendieck shows that $L(E_d/K, T)$ is actually a polynomial which has integral coefficients (see \cite[App. D]{Gross_bsd}).
In particular, it makes sense to consider the multiplicity $\ord_{T=q^{-1}}L(E_d/K, T)\in\Z_{\geq 0}$ of $q^{-1}$ as a zero of this polynomial (see Corollary \ref{coro.expr.rk}).

\subsection{Explicit expression for $L(E_d/K, T)$}  
	Our first main result is the following: 
	
	\begin{theo}\label{theo.Lfunc} For any integer $d\geq 1$ coprime to $q$, consider the elliptic curve $E_d/K$ defined by \eqref{eq.Wmodel}. Set 
	\[Z_{2d} := \begin{cases}
	 \Z/2d\Z\smallsetminus\{0, d/2, d, 3d/2\} &\text{ if $d$ is even},\\
	 \Z/2d\Z\smallsetminus\{0, d\} &\text{ if $d$ is odd}.
	 \end{cases} \]
	The $L$-function of $E_d$ admits the expression
	\begin{equation}\label{eq.Lfunc}
	L(E_d/K, T) = (1-qT)\cdot \prod_{\Mn\in\O_q(Z_{2d})}\left( 1 - \bbeta(\Mn)\cdot T^{|\Mn|}\right),
	\end{equation}
	where $\bbeta(\Mn)$ has been defined in \eqref{eq.defi.jacobi}.
	 \end{theo}
	
	 \begin{rema}
	\begin{enumerate}[(a)] 
	
	\item One can actually deduce from Theorem \ref{theo.Lfunc} an expression of $L(E_d/K,T)$ for \emph{all}  $d\geq 1$, as follows. 
	For any integer $d\geq 1$, write $d=d'p^{e}$ where $e\geq 0$ and $d'\geq 1$ is coprime to $p$. 
	As was pointed out earlier (Remark \ref{rema.alld}) the elliptic curves $E_{d'}$ and $E_{d}$,  being isogenous, share the same $L$-function. 
	Therefore, $L(E_d/K, T)=L(E_{d'}/K, T)$ can also be expressed with the help of Theorem \ref{theo.Lfunc}.
	
	\item In the special case when $2d\mid q-1$, one can choose a character $\cchi:\F_q^\times\to\Qbar^\times$ of exact order $2d$. 
	By Remarks \ref{rema.trivq.act} and \ref{rema.trivq.act2}, the expression \eqref{eq.Lfunc} then simplifies to
	\[L(E_d/K, T) 
	= (1-qT) \cdot \prod_{\substack{1\leq n\leq 2d-1 \\ n\neq d/2, 3d/2}}\left(1- B({\F_q}, \cchi^n)\cdot T\right).\]
	This special case of Theorem \ref{theo.Lfunc} was announced as Theorem \ref{itheo.Lfunc} in the introduction.	
	\end{enumerate}
	\end{rema}
	
	Before we start the proof of Theorem \ref{theo.Lfunc}, we remark that it directly leads to the following `combinatorial' expression of the analytic rank of $E_d/K$:
	\begin{coro}\label{coro.expr.rk}
	For any integer $d\geq 2$ coprime to $q$, 
	\begin{equation}\label{eq.rank}
	\ord_{T=q^{-1}}L(E_d/K, T) 
	= 1+\big|\big\{\Mn\in\O_q(Z_{2d}) \ : \ \bbeta(\Mn) = q^{|\Mn|} \big\}\big|
	\end{equation}
	\end{coro}
	 
		\begin{proof} 
		Let us take a closer look at the factorisation \eqref{eq.Lfunc} of $L(E_d/K, T)$. 
		The factor $1-qT$ clearly contributes for $1$ to the order of vanishing of  $L(E_d/K, T)$ at $T=q^{-1}$. 
		For an orbit $\Mn\in\O_q(Z_{2d})$, the factor $1-\bbeta(\Mn)\cdot T^{|\Mn|}$ vanishes at order $1$  at $T=q^{-1}$ if and only if $\bbeta(\Mn)=q^{|\Mn|}$, and does not vanish otherwise. 
		Summing up these contributions yields \eqref{eq.rank}.
		\ProofEnd \end{proof}
	
	
\subsection{Proof of Theorem \ref{theo.Lfunc}}

In order to prove Theorem  \ref{theo.Lfunc}, it will be useful to have an alternative definition of $L(E_d/K, T)$ at hand.
For an integer $m\geq 1$ and a point $\tau\in\P^1(\F_{q^m})\smallsetminus\{\infty\}$, let $v_\tau$ be the place of $K$ corresponding to $\tau$. 
One may choose  a   polynomial $f_{d,\tau}\in\F_q[t, x]$, monic of degree $3$ in $x$,  such that $y^2=f_{d,\tau}(t, x)$ provides a minimal integral Weierstrass model of $E_d$ at $v_\tau$. 
One can then form the character sum
\[A_d(\tau, q^m) := -  \sum_{x\in\F_{q^m}}  \lambda_{\F_{q^m}}\!\left( f_{d,\tau}(\tau, x)\right),\]
where $\lambda_{\F_{q^m}}$ is the unique character of order $2$ on $\F_{q^m}^\times$. 
By a classical computation, one has
\begin{equation}\label{eq.Atau}
q^m+1 - |\widetilde{(E_d)}_\tau(\F_{q^m})| = q^m - \sum_{x\in\F_{q^m}} \left(1 + \lambda_{\F_{q^m}}\left( f_\tau(x)\right) \right) 
=- \sum_{x\in\F_{q^m}}  \lambda_{\F_{q^m}}\left( f_\tau(x)\right) =A_d(\tau, q^m).  
\end{equation}
Since $E_d$ has split multiplicative reduction at $v=\infty$, we have $q^m+1 - |\widetilde{(E_d)}_\infty(\F_{q^m})|  = q^m+1-|\F_{q^m}|=1$ for all~$m\geq 1$. 
Hence we are led to put $A_{d}(\infty, q^m):=1$ for all~$m\geq 1$. 

\begin{lemm}\label{lemm.expr.Lfunc}
The $L$-function of $E_d/K$ is given by 
\begin{equation}\label{eq.expr.Lfunc}
\log L(E_d/K, T) = \sum_{m=1}^{\infty} \left(\sum_{\tau\in\P^1(\F_{q^m})}A_d(\tau, q^m)\right) \cdot \frac{T^m}{m}.
\end{equation}
\end{lemm}

	\begin{proof} 
	We refer the reader to \cite[\S2.2]{BaigHall} 
	for a detailed proof. 
	The result follows from expanding $\log L(E_d/K, T)$ as a power series in $T$ (from its definition \eqref{eq.def.Lfunc}), rearranging terms and using \eqref{eq.Atau}.
	\ProofEnd\end{proof}

	\begin{proof}[of Theorem \ref{theo.Lfunc}]
	Let $d\geq 1$ be coprime to $q$, and $E_d/K$ be the  elliptic curve defined by \eqref{eq.Wmodel}.
	We deduce the identity \eqref{eq.Lfunc} from the expression for $L(E_d/K, T)$ displayed in Lemma \ref{lemm.expr.Lfunc} by elucidating the double sum on the right-hand side. 
	From the preceding discussion, we already know that, for any integer~$m\geq 1$, we have
	\[ \sum_{\tau\in\P^1(\F_{q^m})} A_d(\tau, q^m)
	= A_d(\infty, q^m) - \sum_{\tau\in\F_{q^m}}\sum_{x\in\F_{q^m}}\lambda_{\F_{q^m}}\big(f_{d, \tau}(x)\big), \] 
	where $A_d(\infty, q^m)=1$. In fact, we can also give a straightforward expression of $A_d(0, q^m)$: 
	if $d$ is odd, then~$E_d$ has additive reduction at $0$ so that $A_d(0, q^n)=0$; 
	whereas when $d$ is even $E_d$ has good reduction, and $y^2=x(x^2-4)$ is a model of the reduction of $E_d$ modulo $v=0$. 
	Hence we have
	\begin{equation}\label{eq.ad.0}
	 A_d(0, q^m)= \begin{cases}
	-\sum_{x\in\F_{q^m}} \lambda_{\F_{q^m}}\big(x(x^2-4)\big) 
	& \text{if $d$ is even, } \\
	0 & \text{if $d$ is odd.}\end{cases}
	\end{equation}
	Now, for any $\tau\in\P^1(\F_{q^m})\smallsetminus\{0, \infty\} = \F_{q^m}^\times$, 
	we have seen in Remark \ref{rem.minimod} that a minimal integral model for the curve $E_d/K$ at $v_\tau$ is 	${y^2= x^3+t^{2d}x^2-4t^{2d}x}$.
	Thus, we have
	 \[\sum_{\tau \in\F_{q^m}^\times} A_d(\tau, q^m) 
	 = -\sum_{\tau\in\F_{q^m}^\times} \sum_{x\in\F_{q^m}} \lambda_{\F_{q^m}}\big(x^3 + \tau^{2d} \cdot x^2 - 4\tau^{2d}\cdot x\big).\]
	 For any $z\in\F_{q^m}^\times$, recall  from \cite[Lemma 2.5.21]{Cohen} that $\big|\{\tau\in\F_{q^m}^\times : \tau^{2d}=z\}\big|$	equals  $\sum_{\chi^{2d}=\trivcar} \chi(z)$, where the sum is over the multiplicative characters $\chi:\F_{q^m}^\times\to\Qbar^\times$ whose order divides $2d$. 
	 Using this identity, one can reindex the outer sum in the last displayed equality: after changing the order of summation, we obtain that
	 \[\sum_{\tau \in\F_{q^m}^\times} A_d(\tau, q^m) 
	 = - \sum_{\chi^{2d}=\trivcar}\left( \sum_{z\in\F_{q^m}^\times} \sum_{x\in\F_{q^m}} \chi(z)\cdot\lambda_{\F_{q^m}} \big(x^3 + z \cdot x^2 - 4z\cdot x\big)\right) = -\sum_{\chi^{2d}=\trivcar} B(\F_{q^m}, \chi),\] 
	 where $B({\F_{q^m}},\chi)$ is the character sum studied in \S\ref{sec.identity}. 
	 To avoid multiple subscripts, for any character $\chi$ on~$\F_{q^m}^\times$ we will denote  in the present proof  the   sum $B(\F_{q^m},\chi)$ by $B(q^m, \chi)$.
	
	 The last displayed identity and the previous remarks about $A_{d}(\infty, q^m)$, $A_{d}(0, q^m)$ lead to:
	\begin{align*} 
	\sum_{\tau\in\P^1(\F_{q^m})} A_d(\tau, q^m) 
	 &= A_d(\infty, q^m) + A_d(0, q^m) - \sum_{\chi^{2d}=\trivcar} B(q^m, \chi) \\
	 &= 1 + A_d(0, q^m) -  \sum_{\substack{\chi^{2d}=\trivcar \\ \chi^4=\trivcar }} B(q^m, \chi) 
	 -\sum_{\substack{\chi^{2d}=\trivcar \\ \chi^4\neq\trivcar }}  B(q^m, \chi). 	
	\end{align*}
	Now, by the first three cases of Proposition \ref{prop.charsum.id}, we can write 	 
	\[ \sum_{\substack{\chi^{2d}=\trivcar \\ \chi^4=\trivcar }} B( {q^m},\chi) 
	 = |\F_{q^m}| + 1 +\sum_{\substack{\chi^{2d}=\trivcar \\ \chi^2\neq \trivcar, \chi^4=\trivcar}} \ja_{\F_{q^m}}(\chi, \chi).\]
	We note that the second sum is empty if $d$ is odd.
	From \eqref{eq.ad.0} and Proposition \ref{prop.charsum.id2}, we then obtain that
	\begin{align*}
	1+A_d(0, q^m) - \sum_{\substack{\chi^{2d}=\trivcar \\ \chi^4=\trivcar }} B({q^m}, \chi)
	&= 	 \begin{cases}
	1 + \sum_{\theta} \ja_{\F_{q^m}}(\theta, \theta) - \left( q^m +1 + \sum_{\theta} \ja_{\F_{q^m}}(\theta, \theta)\right) &\text{if $d$ is even,} \\
	1 + 0 - \left(q^m+1 - 0 \right) & \text{if $d$ is odd.}
	\end{cases} \\
	&= -q^m \text{ in both cases,}
	\end{align*}
	 where the sums on the right-hand side are over characters $\theta$ on $\F_{q^m}^\times$ of exact order $4$.
	
	Regrouping our computations thus far, we have proved that: 
	 \begin{equation*}
	- \sum_{\tau\in\P^1(\F_{q^m})} A_d(\tau, q^m) 
	 = q^m + \sum_{\chi \in X_{2d}(\F_{q^m})} B(q^m, \chi),
	 \end{equation*}
	where the sum on the right-hand side is over the set $X_{2d}(q^m)$ consisting of characters $\chi$ of $\F_{q^m}^\times$ such that $\chi^{2d}$ is trivial and $\chi^4$ is nontrivial.
	Plugging this equality in \eqref{eq.expr.Lfunc}, we obtain that 
	\begin{align*}
	-\log L(E_d/K, T) 
	&= \sum_{m\geq 1} \frac{(qT)^m}{m} 
	+ \sum_{m\geq 1} \left( \sum_{\chi \in X_{2d}(\F_{q^m})} B(q^m, \chi)\right) \cdot\frac{T^m}{m}. \\
	&= -\log\big(1-qT\big)
	+ \sum_{m\geq 1} \left( \sum_{\chi \in X_{2d}(\F_{q^m})} B(q^m, \chi)\right) \cdot\frac{T^m}{m}.
	\end{align*}
	The first sum on the right-hand side clearly leads to the factor $1-qT$ on the right-hand side of \eqref{eq.Lfunc}. 
	There remains to handle the second sum; using Lemma \ref{lemm.reindex}, we can `reindex' the double sum to get
	\[ \sum_{m\geq 1} \left( \sum_{\chi \in X_{2d}(\F_{q^m})} B(q^m, \chi) \right) \cdot\frac{T^m}{m}
	=\sum_{n\in Z_{2d}}  \sum_{s\geq 1} B\big(q^{s\cdot|\Mn|}, \tt_n^{(s)}\big)\cdot\frac{T^{s\cdot |\Mn|}}{{s\cdot |\Mn|}}.\]
	For any $n\in Z_{2d}$ and any $s\geq 1$, Proposition \ref{prop.beta}\eqref{item.beta.HD} yields that $B\big(q^{s\cdot|\Mn|}, \tt_n^{(s)}\big) = \bbeta(n)^s$:
	hence, we have
	\begin{align*}
	 - \log\frac{L(E_d/K, T)}{1-qT} 
	 &=\sum_{n\in Z_{2d}}\sum_{s\geq 1} \bbeta(n)^s \cdot \frac{T^{s\cdot |\Mn|}}{s\cdot|\Mn|}	
	=\sum_{n\in Z_{2d}} \frac{1}{|\Mn|} \sum_{s\geq 1}   \frac{\big(\bbeta(n) \cdot T^{|\Mn|}\big)^s}{s}	\\
	&= - \sum_{n\in Z_{2d}} \frac{\log\big( 1 - \bbeta(n)\cdot T^{|\Mn|}\big)}{|\Mn|}
	= - \sum_{\Mn\in \O_q(Z_{2d})} \sum_{n\in|\Mn|} \frac{\log\big( 1 - \bbeta(n)\cdot T^{|\Mn|}\big)}{|\Mn|} \\
	&= - \sum_{\Mn\in \O_q(Z_{2d})}  \log\big( 1 - \bbeta(\Mn)\cdot T^{|\Mn|}\big).
	\end{align*}
	From which the desired expression for $L(E_d/K, T)$ follows immediately.
	\ProofEnd\end{proof}

 
\section{Berger's construction and the BSD conjecture}
\label{sec.construction}

Let $k=\F_q$ be a finite field of odd characteristic $p$. 
Consider the two rational functions $f,g:\P^1_{/k}\to\P^1_{/k}$ given by:
\[f(u) = \frac{2u-1}{2u(u-1)}, 		\qquad 		g(v) = v(v-1).\]
For any integer $d\geq 1$ coprime to $q$, let $C_d/k$ (resp. $D_d/k$) be a projective smooth model of the curve given affinely by $z^d=f(u)$ (resp. $w^d=g(v)$).
In the notations of \cite{Berger}, the pair $(f,g)$ is of type $[1,1][1,1][1,1][2]$ and we have chosen the value `$a=1/2$' for the parameter in $f$;
see example (6) in~\S4.3 and example (2) in~\S4.4 of \cite{Berger}.
Let us summarise the main results of \cite{Berger} in this case 
(see also \cite{Ulmer_MWJacobians} for a more detailed account). 

By a direct computation, Berger \cite[\S3]{Berger} shows that the smooth projective curve $X_d$ over $K=k(t)$ which is a model of the curve given in affine coordinates by 
\begin{equation}\label{eq.berg0}
f(u)=t^d\cdot g(v)
\end{equation}
has genus $1$.
Clearing denominators, one obtains that $X_d/K$ is given by
\begin{equation}\label{eq.berg1}
\big(u-\tfrac{1}{2}\big)	=t^d\cdot v(v-1) u(u-1)
\end{equation}
The curve $X_d$ obviously admits a $K$-rational point -- namely, $(u,v)=(1/2,0)$ -- so that it is actually an elliptic curve over $K$.

\paragraph{} 
Let us first find a Weierstrass model for this elliptic curve. 
Changing coordinates in~\eqref{eq.berg1} by letting  
\[(x,y)= \left({2t^d(u^{-1}-1)}, \ { 2t^d(u^{-1}-1)}(2v-1)\right):=H(u,v),\]
we find that $X_d/K$ admits the Weierstrass model
\begin{equation}\notag{} 
X_d:\qquad y^2=x\cdot\big(x^2+t^{2d}\cdot x-4t^{2d}\big).
\end{equation}
Moreover, the `change of coordinates' map $H: X_d\to E_d$ 
is birational,  
with birational inverse given by $(x,y)\mapsto (u,v) =\left( {2t^d}\cdot({x+2t^d})^{-1}, \ {(x+y)}\cdot({2x})^{-1}\right)$.
Therefore the elliptic curves $X_d$ and~$E_d$  are actually $K$-isomorphic. 


\paragraph{} 
We now describe in more details the geometry of the situation. 
For any $d\geq 1$ coprime to $q$,  consider the projective surface\footnote{The surface $S_d$ may be singular.} $S_d$ over $k$ defined affinely by  \eqref{eq.berg0}. 
There is a dominant rational map $\rho: C_d\times D_d\dashrightarrow S_d$, given in affine coordinates by $ \rho: \big((u,z),(v,w)\big) \mapsto (t =z/w ,u,v)$. 
The surface $S_d$ also admits a natural morphism $\pi_0:S_d\to\P^1$ extending the projection $(t,u,v)\mapsto t$. 
The generic fiber of $\pi_0:S_d\to\P^1$ is clearly the elliptic curve $X_d/K$, \emph{a.k.a.} $E_d/K$ by the previous paragraph.

Besides, the group $\mu_d$ of $d$th roots of unity in $\bar{k}$ acts on $C_d\times D_d$, \emph{via} $\zeta \cdot (u,z,v,w) := (u,\zeta z, v, \zeta w)$. 
We denote the quotient surface by $(C_d\times D_d)/\mu_d$ and by $\sigma: C_d\times D_d \to (C_d\times D_d)/\mu_d$ the quotient morphism%
\footnote{The quotient surface $(C_d\times D_d)/\mu_d$ may be singular.}.
Given its definition, it is clear that the map $\rho$ above factors through 
$\sigma$. 
Since both $\rho$ and $\sigma$ have degree~$d$, the induced rational map $
\tilde{\rho}: (C_d\times D_d)/\mu_d \dashrightarrow S_d$ is  birational. 
By blowing up the singular points of $(C_d\times D_d)/\mu_d$, one can resolve it into a regular surface (in a minimal way). 
We denote by $\Ecal_d/k$ that surface: by construction,  the successive blow-ups provide a birational map $b: \Ecal_d\dashrightarrow (C_d\times D_d)/\mu_d$. 
Composing this map with $\pi_0\circ\tilde{\rho}$, one endows $\Ecal_d$ with a morphism $\pi:\Ecal_d\to \P^1$ which `extends' $\pi_0:S_d\to\P^1$. 
In particular, the generic fiber of $\pi$ is the same as that of $\pi_0$, \ie{} it is $E_d/K$. 
Hence, $\pi: \Ecal_d\to\P^1$ is the minimal regular model of $E_d/K$. 
On the other hand, note that composing the birational inverse of  $b$ with~$\sigma$ yields a dominant rational map $w_d: C_d\times D_d \dashrightarrow \Ecal_d$.
Here is a diagram describing the situation:
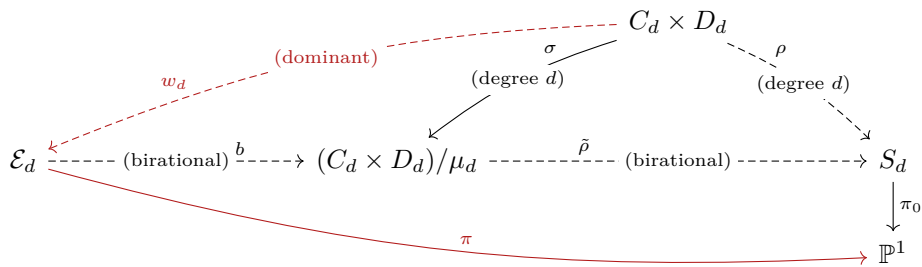
\begin{center}
\begin{tikzcd}
&&&&&&C_d\times D_d 
\ar[ddllllll, bend right=10, dashed, color=myemph, "\text{(dominant)}" description, "w_d"' near end]
\ar[ddll, "\sigma"'  near start, "\text{(degree $d$)}" description, bend right=10]
\ar[ddrr, bend left=10, dashed, "\rho" near start, "\text{(degree $d$)}" description]&&&\\
\\\
\mathcal{E}_d 
\ar[rrrr, dashed, "b" near end,  "\text{(birational)}" 
description]
\ar[drrrrrrrr, bend right=9, color=myemph, "\pi"] &&&& 
(C_d\times D_d)/\mu_d 
\ar[rrrr, dashed, "\tilde{\rho}" near start, "\text{(birational)}" description] &&& & 
{S}_d
\ar[d, "\pi_0"] \\ 
&&&&&&&& \mathbb{P}^1 \\
\end{tikzcd}
\end{center}
Summarising the discussion above, $\pi : \mathcal{E}_d\to \mathbb{P}^1$ is an elliptic surface (over $k$) with generic fiber $E_d/K$, and the surface $\Ecal_d$ is dominated by a product of curves. 
 
The most striking feature of Berger's construction is the following: 

\begin{theo}[Berger]\label{theo.BSD}
 Let $d\geq 1$ be an integer. 
The elliptic curve $E_d/K$ given by \eqref{eq.Wmodel} satisfies the BSD conjecture. 
In particular, one has
\begin{equation*}
	\rk E_d(K) = \ord_{T=q^{-1}}L(E_d/K, T).
\end{equation*}
\end{theo}

\begin{proof} 
We only sketch the argument and refer the interested reader to \cite[\S2]{Berger} or~\cite{Ulmer_MWJacobians} for a detailed proof. 
As before, we write $d=d'p^{e}$ where $e\geq 0$ and $d'\geq 1$ is coprime to $q$.
The truth of the BSD conjecture is invariant under isogeny (see \cite[Chap. I, Thm. 7.3]{Milne_ArithDual_2nd}): 
since $E_d$ and $E_{d'}$ are isogenous (see Remark~\ref{rema.alld}), it  suffices to prove the Theorem in the case where $d=d'$ is coprime to $q$.

Hence we assume that $d\geq 1$ is coprime to $q$ and we denote by $\pi:\Ecal_d\to\P^1$ the minimal regular model of~$E_d/K$. 
Proving the `rank part' of the BSD conjecture for $E_d/K$ is equivalent to proving the Tate conjecture for the surface $\Ecal_d/\F_q$ (see \cite{Tate_BSD}).  
That conjecture is known to hold for surfaces which admit a dominant rational map from a product of curves (see \cite{Tate_conj}).
By the discussion above, $\Ecal_d$ admits such a map $w_d:C_d\times D_d\dashrightarrow \Ecal_d$, 
hence the Theorem. 
\ProofEnd\end{proof}

More generally, the construction described by Berger in \cite[\S2-\S3]{Berger}  provides many examples of Kummer families of elliptic surfaces which are dominated by products of curves.
For all these elliptic surfaces, the Tate conjecture holds.
Therefore, for all the corresponding elliptic curves over $K$, the BSD conjecture is known to hold. 
Berger's construction thus provides a large range of examples where Ulmer's theorem about unbounded analytic ranks (see Theorem \ref{theo.Ulmer}) can be unconditionally translated into an `unbounded algebraic rank' result.


\section{Supersingular integers and ranks}
\label{sec.supersing}

Let $\F_q$ be a finite field of characteristic $p\geq 3$ and $K=\F_q(t)$. 
For any integer $d\geq 1$, consider the elliptic curve $E_d/K$ defined by \eqref{eq.Wmodel}. For the remainder of the article, we let
\begin{equation}
\rank(d):=\rk\big(E_d(K)\big).
\end{equation}
The goal of the next section will be to study the behaviour of the sequence $(\rank(d))_{d\geq 1}$.
In the present section, we first use the above results to give a combinatorial expression of $\rank(d)$ for certain values of $d$.

\subsection{Supersingular integers and Jacobi sums}

	Recall that an integer $D\geq 1$ is called \emph{supersingular}
	if $D$ divides $q^a +1$ for some integer $a\geq 1$. 
	We denote by $\Ss_q$ the set of supersingular integers.  
	Since $q$ is odd, note that $2d\in\Ss_q$ for any odd supersingular $d$. 
	
	Supersingular integers are of interest to us because of the following results:
	
	\begin{lemm}[Ulmer]\label{lemm.ulmer}
	Let $D\geq 1$ be an \emph{even} supersingular integer. 
	Then the order of $q$ modulo $D$ is even. 
	For any orbit $\Mn\in\O_q(\Z/D\Z\smallsetminus\{0, D/2\})$, its length $|\Mn|$ is even and $D$ divides $n(q^{|\Mn|/2}+1)$.
	
	Moreover, if $D=q^a+1$ for some $a\geq 1$, the order of $q$ modulo $D$ equals $2a$.
	\end{lemm}
	For proofs of this lemma, we refer the reader to \cite[Lemma 8.2]{Ulmer_LargeRk} or to \cite[Lemmes 2.4.1 \& 2.4.2]{Griffon_PHD} where a more detailed argument is given.
	
	\begin{lemm}[Shafarevich -- Tate / Ulmer]\label{lemm.shaftate} 
	Let $\F_{Q^2}/\F_Q$ be a quadratic extension of finite fields of odd characteristic. 
	Let $\chi_1, \chi_2:\F_{Q^2}^\times\to\Qbar^\times$ be  nontrivial characters on $\F_{Q^2}^\times$ such that $\chi_1\chi_2$ is nontrivial. 
	If the restrictions of $\chi_1$ and $\chi_2$ to  $\F_Q^\times$ are trivial, we have 
	\begin{equation}\notag
	\ja_{\F_{Q^2}}(\chi_1, \chi_2)=-Q=-\sqrt{Q^2}. 
	\end{equation}
	\end{lemm}
	The proof  can be found in \cite[Lemma]{ShaTate_Rk} or \cite[Lemma 8.3]{Ulmer_LargeRk}.
	In both references, the result is phrased in terms of Gauss sums, but is easily translated to the desired identity, see \cite[Prop. 2.5.14]{Cohen}. 
	
	\medskip
	From the above lemmas, we deduce the following result about the sums $\bbeta(\Mn)$ defined in \S\ref{sec.bbeta}:
	
	\begin{theo}\label{theo.ShaTate}
	Let $d\geq 1$ be a integer such that $2d\in\Ss_q$. 
	Then, for all $\Mn \in \O_q(Z_{2d})$, we have $\bbeta(\Mn)=q^{|\Mn|}$.  	
	\end{theo}
	 	
		\begin{proof} 
		Let $n\in Z_{2d}$: 
		by Lemma \ref{lemm.ulmer} (applied to $D=2d$), the length $|\Mn|$ of its orbit $\Mn\in\O_q(Z_{2d})$ is even.
		In particular, $Q:=q^{|\Mn|/2}$ is an integral power of $q$, and 	the finite field extension $\F_{Q^2}/\F_{Q}$ is   quadratic.
		By~Lemma \ref{lemm.ulmer} again, $(Q+1)n/e$ is an integer and  by construction of  $\tt_n:\F_{Q^2}^\times\to\Qbar^\times$, we deduce that 
		\[\forall x\in\F_Q^\times, \quad
		\tt_n(x) 
		= \tt(x)^{(Q^2-1)n/e} 
		= \tt_n(x^{Q-1})^{(Q+1)n/e} 
		= \tt_n(1)^{(Q+1)n/e} = 1.\]
		The character $\cchi_1=\tt_n$ is thus a nontrivial character of $\F_{Q^2}^\times$ whose restriction to $\F_Q^\times$ is trivial; and the character $\cchi_2=\tt_n^{-1}$ obviously has the same property. 
		Consider now the character $\cchi_3=\lambda_{\F_{Q^2}}\tt_n^2$: it cannot be trivial otherwise $\tt_n$ would have order dividing $4$, which does not happen for $n\in Z_{2d}$, see \S\ref{sec.char}. 
		Nonetheless, $\cchi_3$ has trivial restriction to $\F_Q^\times$,  being the product of two characters whose restrictions to $\F_Q^\times$ are trivial,  (the restriction of $\lambda_{\F_{Q^2}}$ is trivial because any element of $\F_Q^\times$ becomes a square in $\F_{Q^2}^\times$).
		Besides, note  that $\cchi_1^2=\tt_n^2$ and $\cchi_2\cchi_3 = \lambda_{\F_{Q^2}}\tt_n$ are nontrivial. 
		
		We can thus apply Lemma \ref{lemm.shaftate} and  obtain that the  Jacobi sums $\ja_{\F_{Q^2}}(\cchi_1, \cchi_1)$ and $\ja_{\F_{Q^2}}(\cchi_3, \cchi_2)$ both equal~$-Q$.
		On the other hand, since the restriction of $\tt_n$ to $\F_Q^\times$ is trivial and since $4\in\F_Q^\times$, we have $\tt_n^2(4)=1$. 
		Therefore, by its definition (see \eqref{eq.defi.jacobi}) the sum $\bbeta(n)$ satisfies:
		\[\bbeta(n) 
		=  \tt_n^{2}(4) \cdot \ja_{\F_{Q^2}}(\tt_n, \tt_n) \ja_{\F_{Q^2}}(\lambda_{\F_{Q^2}}\cdot\tt_n^2, \tt_n^{-1})
		= 1\cdot\ja_{\F_{Q^2}}(\cchi_1, \cchi_1)\ja_{\F_{Q^2}}(\cchi_3, \cchi_2) = Q^2 = q^{|\Mn|}, \]
		as was to be shown.
		\ProofEnd\end{proof}
	 	 
	 Let us combine several of the results obtained so far. 
	 For any integer $d\geq 1$ coprime to $q$, the BSD conjecture (Theorem \ref{theo.BSD}) ensures that $\rank(d)= \ord_{T=q^{-1}}L(E_d/K, T)$. 
	 Besides, Corollary \ref{coro.expr.rk} yields an expression for $\ord_{T=q^{-1}}L(E_d/K,T)$ in terms of the number of $\Mn\in\O_q(Z_{2d})$ such that $\bbeta(\Mn)=q^{|\Mn|}$. 
	 If we further assume that $d\geq 1$ is such that $2d\in\Ss_q$,  the previous Theorem then implies:
	
	\begin{coro}\label{coro.rk.supersing} 
	Let $d\geq 1$ be an integer such that $2d\in\Ss_q$.	
	We have $\rank(d) =1+ |\O_q(Z_{2d})|$.
	\end{coro}
	
	\begin{rema} 
	Let $d\geq 1$ be an integer such that $2d\in\Ss_q$. 
	Theorem \ref{theo.ShaTate} actually shows that the $L$-function of $E_d/K$ admits the following expression:	
	\[L(E_d/K, T) = (1-qT)\cdot \prod_{\Mn\in\O_q(Z_{2d})}\left(1-(qT)^{|\Mn|}\right).\]
	\end{rema}

\subsection{Structure of $\O_q(Z_{2d})$} 
\label{subsec.struct}

	In order to make Corollary \ref{coro.rk.supersing} more explicit, we describe in more detail the structure of $\O_q(Z_{2d})$.
	Let $d\geq 1$ be any integer coprime to $q$ and define the set $Z_{2d} \subset \Z/2d\Z$ as in~\S\ref{sec.q.act}.
	For any divisor $e>2$ of $2d$, consider the subset
	\[ Y_e:=\left\{n\in Z_{2d} : \gcd(n,2d)=2d/e\right\} \subset Z_{2d};\]
	since $\gcd(2d,q)= 1$, this subset 
	is stable under the action of $q$. 
	It is then clear that $Z_{2d}$ is the disjoint union of the $Y_e$, and that the orbit set $\O_q(Z_{2d})$ is the disjoint union of the sets $\O_q(Y_e)$ as $e$ runs through the divisors of~$2d$ which are $>2$.
	
	We denote by $\phi$ Euler's totient function\footnote{%
	We use the convention that $\phi(1)=0$, so that $\phi(n)=|(\Z/n\Z)^\times|$ for all $n\geq 1$.%
	}, and 
	by $o_q(n)$ the multiplicative order of $q$ modulo $n$, for any integer $n\geq 2$ coprime to $q$. 
	Since $o_q(n)$ is the order of the subgroup generated by $q$ in $(\Z/n\Z)^\times$, we have $o_q(n)\mid \phi(n)$. 
	In the notations of the previous paragraph, for any $e\mid 2d$ with $e>2$, all orbits $\Mn\in\O_q(Y_e)$ have length $|\Mn|=o_q(e)$. 
	Notice also that $Y_e$ is in bijection with $(\Z/e\Z)^\times$. 
	These two observations show that 
	\[|\O_q(Y_e)| = |(\Z/e\Z)^\times|/o_q(e) = \phi(e)/o_q(e).\]
	
	This argument proves: 
	
	\begin{lemm}\label{lemm.struct.Oq}
	Let  $d\geq 1$ be an integer coprime to $q$. Then
	$\displaystyle |\O_q(Z_{2d})| = \sum_{\substack{e\mid 2d \\ e>2}} \frac{\phi(e)}{o_q(e)}$. 
	\end{lemm}
	
	Following the notations of \cite{PomeShpa}, for any integer $D$ coprime to $q$, we let 
	 \[ I_q(D) := \sum_{e\mid D} \frac{\phi(e)}{o_q(e)}.\]
	
	For all $d\geq 1$ such that $2d\in\Ss_q$, let us combine the results of Corollary \ref{coro.rk.supersing} and of the above Lemma: noticing that $e=1$ does not contribute and that $\phi(2)/o_q(2)=1$, we obtain that 
	\begin{equation}\label{eq.ineq.rk1}
	\rank(d)  = 1 +|\O_q(Z_{2d})|  =I_q(2d).
	\end{equation}
	Using the obvious fact that, if $E\geq 1$ is a divisor of a prime-to-$q$ integer $D$ one has $I_q(D)\geq I_q(E)$, we further obtain that 
	\begin{equation}\label{eq.ineq.rk2}
	\rank(d) \geq I_q(d).
	\end{equation}
	This inequality will be be useful in the next section.


 \section{Unbounded rank for $\{E_d\}_{d\geq 1}$}
 \label{sec.ubnbdrk}
 
In this section, we prove our second main result: 
\begin{theo}\label{theo.unbounded.rank} 
Let $\F_q$ be a finite field of odd characteristic, and $K:=\F_q(t)$.
For any integer $d\geq 1$, consider as above the elliptic curve $E_d$ defined by \eqref{eq.Wmodel} over $K$.
One has 
\[\limsup_{d\geq1}\left(\rk E_d(K) \right) = +\infty.\] 	
\end{theo}

We give two proofs of this result.
In the first one, we exhibit various sequences $(d_n)_{n\geq 1}$ of integers such that $\rank(d_n)\to\infty$ as $n\to\infty$. 
The second proof shows a stronger (but less explicit) result: we deduce from the work of Pomerance and Shparlinski that the average of $\rank(d)$ is unbounded.
We conclude the section by an amusing observation about a link between prime numbers and the sequence $(\rank(d))_{d\geq 1}$. 

\begin{rema} 
In view of Remark \ref{rema.kummer}, we note the following other way of stating Theorem \ref{theo.unbounded.rank}.
Let $E=E_1$ be the elliptic curve over $\F_q(t)$ defined by  $y^2 = x(x^2 + t^2\cdot x^2 -4t^2)$.
For any integer $d\geq 1$, 
consider the finite Kummer extension $K_d:=\F_q(t^{1/d})$ of $K$. 
Theorem \ref{theo.unbounded.rank} asserts that $ \limsup_{d\geq 1}\left( \rk E(K_d) \right)  = +\infty$.
\end{rema}

\subsection{Unbounded ranks (I)}
\label{sec.unbounded}

Recall that $\Ss_q$ denotes the set of supersingular integers. 
Let us start by constructing some special sequences of integers:
\begin{lemm}\label{lemm.seq.int}  
 For all $n\geq 1$, consider $d^e_n:=q^n+1$ and $d^o_n:=\sum_{i=0}^{2n} (-q)^i$. 
\begin{enumerate}[\rm(i)]
\item\label{item.one}
\setlength\itemsep{0em}
 $d^e_n$ is an even integer such that $d^e_n\in\Ss_q$. Moreover, one has $I_q(d^e_n)\geq \log\sqrt{q}\cdot d^e_n/\log d^e_n$. 
\item\label{item.odd}
$d^o_n$ is an odd integer such that $2d^o_n\in\Ss_q$. Moreover, one has $I_q(2d^o_n)\geq \log\sqrt{q}\cdot d^o_n/\log d^o_n$. 
\end{enumerate}
\end{lemm}

\begin{proof} 
It is obvious that $d^e_n=q^n+1$ is even and supersingular; 
by Lemma \ref{lemm.ulmer}, the multiplicative order of $q$ modulo~$d^e_n$ equals $2n$. 
Moreover, for any divisor $e\mid d^e_n$, we have $o_q(e)\mid o_q(d^e_n) = 2n$ so that, in particular, $o_q(e) \leq o_q(d^e_n)$. 
We obtain the chain of inequalities:
\[I_q(d^e_n) = \sum_{e\mid d^e_n} \frac{\phi(e)}{o_q(e)} 
\geq \frac{1}{o_q(d^e_n)}\cdot \sum_{{e\mid d^e_n}}  \phi(e)  =\frac{d^e_n}{o_q(d^e_n)} = \frac{d^e_n}{2n}.\]
By construction, one has $n = \log(d^e_n-1)/\log q \leq \log d^e_n /\log q$ and we immediately deduce that $I_q(d^e_n)\geq \log\sqrt{q}\cdot d^e_n\log d^e_n$, hence \eqref{item.one} is proved.

Being the sum of an odd number of odd integers, $d^o_n$ is clearly odd, and we note that $q^{2n+1}+1=\frac{q+1}{2}\cdot 2d^o_n$. 
This identity shows that $2d^o_n$ is supersingular  and that $q^{2n+1}\equiv -1 \bmod{2d^o_n}$, so that $o_q(2d^o_n) \leq 2(2n+1)$.
As before, for any divisor $e\mid 2d^o_n$, we have $o_q(e) \leq o_q(2d^o_n)$  and, by an argument similar to the above, we obtain  that $I_q(2d^o_n)\geq d^o_n/(2n+1)$. 
From the identity $d^o_n = {(q^{2n+1}+1)}/{(q+1)}$, straightforward estimates imply that $d^o_n/(2n+1)\geq \log \sqrt{q} \cdot d^o_n/\log d^o_n$, which proves \eqref{item.odd}.
\ProofEnd\end{proof}
 
Putting together our results so far, we obtain: 
 
 \begin{theo}\label{theo.unbounded} 
 Let $\F_q$ be a finite field of odd characteristic.
 For any integer $d\geq 1$, consider the elliptic curve $E_d$ over $K=\F_q(t)$ defined by \eqref{eq.Wmodel}. 
 There is an infinite sequence of integers $(d_n)_{n\geq 1}$ such that 
\[\rank(d_n)=\rk\big(E_{d_n}(K)\big) \gg_{q} \frac{d_n}{\log d_n}.\]
More precisely,
\begin{enumerate}[\rm(a)]
 \setlength\itemsep{0em}
\item\label{item.rank.unbounded.a}
 there is an infinite sequence of odd integers $(d^o_n)_{n\geq 1}$ such that $\rank(d^o_n) \geq\displaystyle \log\sqrt{q}\cdot{d^o_n}/{\log d^o_n}$.  
\item\label{item.rank.unbounded.b}
 there is an infinite sequence of even integers $(d^e_n)_{n\geq 1}$ such that $\rank(d^e_n) \geq\displaystyle \log\sqrt{q}\cdot{d^e_n}/{\log d^e_n}$ 
\end{enumerate}
 \end{theo}

\begin{proof} 
For any integer $d\geq 1$ such that $2d\in\Ss_q$, we know from \eqref{eq.ineq.rk1} that $\rank(d) = \rk\big(E_d(K)\big)=I_q(2d)$.

By Lemma \ref{lemm.seq.int}\eqref{item.odd}, there exists an infinite sequence of odd integers $d^o_n:=\sum_{i=0}^{2n}(-q)^i$ such that $2d^o_n$ is supersingular and $I_q(2d^o_n) \geq \log\sqrt{q}\cdot d^o_n/\log d^o_n$ for all $n\geq 1$. 
This proves  assertion \eqref{item.rank.unbounded.a} of the Theorem.

Next, let $d^e_n:=q^n+1$ for any $n\geq 1$; note that $d^e_n$ is even and supersingular. 
By the small Lemma~\ref{lemm.rk.mult} below, one has $\rank(d^e_n) \geq \rank(d^e_n/2)$. 
Together with \eqref{eq.ineq.rk1} applied for $d=d^e_n/2$,  Lemma \ref{lemm.seq.int}\eqref{item.one} then implies that
\[\rank(d^e_n) \geq \rank(d^e_n/2) = I_q(d^e_n) 
 \geq  \log\sqrt{q}\cdot {d^e_n}/{\log d^e_n}.\]
This proves the Theorem, and concludes the first proof for Theorem \ref{theo.unbounded.rank}.  
\ProofEnd\end{proof}
 
\begin{lemm}\label{lemm.rk.mult}
Let $d\geq 1$ be an integer.
Then $\rank(md)\geq\rank(d)$ for all integers $m\geq 1$.
\end{lemm}
\begin{proof} 
For a given $m\geq 1$, we let $K_m:=\F_q(t^m)$ be the subfield of $\F_q(t)=K$ consisting of rational functions in the variable $t^m$.
Since $K_m\subset K$, it is clear that $E_{md}(K_m)\subseteq E_{md}(K)$.
On the other hand, it is obvious that $E_{md}(K_m)\simeq E_d(K)$ as abelian groups.
Hence, $E_{d}(K)$ is isomorphic to a subgroup of $E_{md}(K)$.
\ProofEnd\end{proof}

\begin{rema}\label{rema.unbdd}
\begin{enumerate}[(a)]
\item 
Given that the BSD conjecture holds for the curves $E_d$ (see Theorem \ref{theo.BSD}), the lower bound in Theorem~\ref{theo.unbounded}\eqref{item.rank.unbounded.b} means that Ulmer's lower bound on the rank (Theorem \ref{theo.Ulmer}) still holds for the sequence $\{E_d\}_{d\geq 1}$, even though its `parity hypothesis' fails to be satisfied  (see \S\ref{sec.Ulmer}).

\item\label{item.brumer} 
Let us further compare the lower bounds in Theorem \ref{theo.unbounded}  to Brumer's upper bound on the rank. 
Proposition~6.9 in~\cite{Brumer} states that 
 \begin{equation}\label{eq.Brumer}
 	\rk\big(E(K)\big)\leq \log\sqrt{q}\cdot \frac{\deg\cond_{E/K}}{\log\deg\cond_{E/K}}\cdot\big(1+o(1)\big)\quad (\text{as }\deg\cond_{E/K}\to\infty),
 \end{equation}
 for any nonisotrivial elliptic curve $E/K$.
Applying this bound to $E=E_d$ and plugging in our computation of $\deg \cond_{E_{d}/K}$ yields that 
$\rank(d') \leq \log \sqrt{q}\cdot  \frac{2d'}{\log d'}\cdot (1+o(1))$, as $d'\to\infty$ runs through integers coprime to $q$.
Combining this to the lower bounds on $\rank(d^x_n)$   provided by Theorem~\ref{theo.unbounded}\eqref{item.rank.unbounded.a}-\eqref{item.rank.unbounded.b} for $x\in\{o,e\}$ and all $n\geq 1$, we deduce that 
\[\log\sqrt{q}\cdot \frac{d^x_n}{\log d^x_n} 
\leq \rank(d^x_n) \leq  
2\log \sqrt{q}\cdot  \frac{d^x_n}{\log d^x_n}\cdot (1+o(1))\qquad(\text{as }n\to\infty).\]
In other words, the subsequences  $\{E_{d^x_n}\}_{n\geq 1}$ for $x\in\{o,e\}$ provide examples where Brumer's bound~\eqref{eq.Brumer} is optimal, up to a small absolute constant. 
 \end{enumerate}
 \end{rema}

\subsection{Unbounded ranks (II)}  
\label{sec.unbounded.redux}
In \cite{PomeShpa}, Pomerance and Shparlinski study the average behaviour of the Mordell--Weil ranks of the elliptic curves of \cite{Ulmer_LargeRk}.
In the introduction to their paper, they note that it would be interesting to extend their result to other families of elliptic curves.
Motivated by this remark, we investigated the average rank of the sequence $\{E_d\}_{d\geq 1}$ under consideration here.
Here is the outcome of this investigation:

\begin{theo}\label{theo.unbounded.redux}
There exists an absolute constant $\alpha>1/2$ such that, for all big enough $x$ (depending on~$p$) one has
\begin{equation*}
\frac{1}{x}\sum_{1\leq d\leq x} \rank(d) \geq x^\alpha.
\end{equation*}
In particular, the average rank of $\{E_d(K)\}_{d\geq 1}$ is unbounded.
\end{theo}

It turns out that this result follows almost directly from the constructions in \cite{PomeShpa}   (specifically, the proof of Theorem 1 there). 
For the convenience of the readers, we sketch a proof nonetheless; more details can be found on pp. 33 -- 35 of \emph{loc.\,cit.}. 

\begin{proof}
Let $x$ be a large real number (depending on $p$). 
In their paper (see p. 35 \emph{loc.\,cit.}), Pomerance and Shparlinski show that there exist  an absolute constant $\alpha>1/2$ and a set $S_x$ of integers, with the following properties:
\begin{multicols}{2}
\begin{enumerate}[(i)]
\setlength\itemsep{0em}
\item For all $d\in S_x$, $x^{1+o(1)}\leq d \leq x$, 
\item  $|S_x|\geq x^{\alpha + o(1)}$, 
\item For all $d\in S_x$, $o_q(d) \leq x^{o(1)}$, 
\item For all $d\in S_x$, $d$ is odd and $d$ is supersingular.
\end{enumerate}
\end{multicols}
\noindent
The construction of $S_x$ is quite subtle (\cf{} pp. 32--34 \emph{loc.\,cit.}) and relies on techniques from analytic number theory. 
We do not go into  details and refer the  reader to the paper \cite{PomeShpa} for details. 

Since it is known that $\phi(d)\gg d^{1+o(1)}$ for all $d\geq 1$, we deduce from properties (iii) and (i) of $S_x$ that 
\[\forall d\in S_x, \qquad 
I_q(d) \geq\frac{\phi(d)}{o_q(d)} 
\geq \frac{d^{1+o(1)}}{x^{o(1)}} \geq x^{1+o(1)}.\]
Hence, by (ii),
\[\frac{1}{x}\sum_{d\in S_x} I_q(d)  \geq \frac{|S_x|}{x} \cdot x^{1+o(1)} \geq x^{\alpha +o(1)}.\]
By (iv), any $d \in S_x$ is odd and supersingular, so that $2d$ is supersingular for all $d\in S_x$. 
Thus, $\rank(d)\geq I_q(d)$ by our inequality \eqref{eq.ineq.rk2}.
From here, since $\rank(d)\geq 1$, we deduce that
\begin{align*}
\frac{1}{x}\sum_{1\leq d\leq x} \rank(d) 
&\geq \frac{1}{x}\sum_{d\in S_x} \rank(d) 
\geq \frac{1}{x}\sum_{d\in S_x} I_q(d)  
\geq x^{\alpha +o(1)},
\end{align*}
which concludes the proof of the Theorem. 
\ProofEnd \end{proof}

\begin{rema}
\begin{enumerate}[(a)]
\item 
	Brumer has shown that the average rank of elliptic curves over $K$ is bounded. 
	More precisely, consider the sequence $\Escr\!\ell\!\ell$ of \emph{all nonisotrivial} elliptic curves over $K$ ordered by their naive height $h$: Theorem 7.11 in  \cite{Brumer} states that 
	\[	\frac{1}{\left|\big\{E\in \Escr : h(E)\leq B\big\}\right|} \sum_{\substack{E\in\Escr \\ h(E)\leq B}} \rk\big(E(K)\big) \leq 2.3 +o(1) 
	\qquad (\text{as } B\to\infty).\]
	Comparing this result to our Theorem \ref{theo.unbounded.redux}, one sees that the sequence $\{E_{d}\}_{d\geq 1}$ must be a very 
	`thin' subsequence of $\Escr\!\ell\!\ell$. 
	It is also a quite special subsequence, in the sense that the average behaviour of its rank is atypical.
\item 
	From Theorem \ref{theo.unbounded.redux}, one easily deduces the following statement: for $x$ large enough, there are at least $x^\alpha$ integers $d\in[1, x]$ such that $\rank(d)\gg_q \sqrt{d}/\log d$. 
	In particular, `large' ranks are relatively `common' in the sequence $\{E_d\}_{d\geq 1}$.
	 But actually, adapting Theorem 2 of \cite{PomeShpa} to the case at hand would show that a much stronger statement holds: this would prove that, for all $\epsilon>0$, one has 
	 \[ \frac{1}{x}\cdot \left|\big\{ d\in [1,x] : \rank(d)\geq (\log d)^{(1/3-\epsilon)\log\log\log d} \, \big\}\right| = 1-o_{p, \epsilon}(1) \qquad (\text{as }x\to\infty).\]
	More vaguely, this tells us that the rank $\rank(d)$ is `reasonably large' for almost all integers $d$.
	Again, the proof of \cite[Theorem 2]{PomeShpa} is quite subtle and we refer the reader to \cite[pp. 36--39]{PomeShpa} for  details.  
\end{enumerate}
\end{rema}

\subsection{An amusing fact}
\label{sec.amusing}

We would like to conclude this paper by the following observation.
In this section, we restrict to the case where $q=p$ is a prime number $p\geq 3$ and consider the elliptic curves $E_d$ defined by \eqref{eq.Wmodel} over $K=\F_p(t)$.

When $d=1$ or $2$, we know from Theorem \ref{theo.Lfunc} that $L(E_d/K, T)=1-qT$.
In particular, the BSD conjecture (Theorem \ref{theo.BSD}) here yields that $\rank(1)=\rank(2)=1$. 
Since, for all $m\geq 1$, the $p^m$th power Frobenius provides an isogeny $E_1\to E_{p^m}$ (resp. $E_2\to E_{2p^m}$) and since isogenous curves have the same Mordell--Weil rank, the groups  $E_{p^m}(K)$ and $E_{2p^m}(K)$ also have rank one.
Hence, we deduce that there exist infinitely many integers $d\geq 1$ such that $E_d$ has Mordell--Weil rank $1$  (namely $1, 2, p, 2p, p^2, 2p^2, p^3, \dots$). 
Combined with Corollary \ref{coro.positive.rk} -- which implies that $\rank(d)\geq 1$ for all $d\geq 1$, we have just shown that $\liminf_{d\geq 1}\rank(d)=1$.

One can generalise the above argument, as follows:

\begin{theo}\label{theo.amusing}
Assume that there exists a prime number $\ell\neq 2,p$ such that $p$ generates $(\Z/\ell^2\Z)^\times$.
For any odd integer $R\geq 1$, there are infinitely many integers $d\geq 1$ such that 
\[\rk\big(E_d(\F_p(t))\big)=R.\]
\end{theo}

Given a prime $p\geq 3$, and assuming the existence of a prime $\ell\neq 2,p$ such that $p$ generates $(\Z/\ell^2\Z)^\times$, 
this Theorem gives a third proof of the unboundedness of $d\mapsto\rk\big(E_d(\F_p(t)\big)$ (Theorem \ref{theo.unbounded.rank}).

\begin{proof}  
The case $R=1$ has already been handled above (without any assumption about primes), hence we can assume that $R\geq 3$ and  let $r:=(R-1)/2\geq 1$.

Pick an odd prime $\ell$ such that the subgroup generated by $p$ modulo $\ell^2$ is the whole of $(\Z/\ell^2\Z)^\times$.
A classical argument shows that  $p$ then also generates $(\Z/\ell^b\Z)^\times$ for all $b\geq 1$ (see \cite[Ex. I.6.2]{Hin_eng} for instance); 
and it is easily seen that, for all $b\geq 0$,  $p$ generates $(\Z/2\ell^b\Z)^\times$  as well. 
Therefore, for this choice of $\ell$, one has  $o_p(\ell^b)=\phi(\ell^b)$ and $o_p(2\ell^b)=\phi(2\ell^b)$. 

Since $p$ generates $(\Z/2\ell^r\Z)^\times$, there  exists an integer $a\geq 1$ such that~${p^a\equiv -1 \bmod{2\ell^r}}$ \ie{}, such that $2\ell^r$ divides $p^a+1$. 
Hence the integer $2\ell^r$ is  supersingular. 
Moreover, we have
\[I_p(2\ell^r) 
= \sum_{e\mid 2\ell^r} \frac{\phi(e)}{o_p(e)}
= \sum_{1\leq b\leq r} \frac{\phi(\ell^b)}{o_p(\ell^b)} + \sum_{0\leq b\leq r} \frac{\phi(2\ell^b)}{o_p(2\ell^b)} 
= r+(r+1)=R.\]
We deduce from Corollary \ref{coro.rk.supersing} and \eqref{eq.ineq.rk1} that one has $\rk\big(E_{\ell^r}(K)\big) = I_p(2\ell^r) =R$.

Now put $d_m := p^m\ell^r$ for all $m\geq 0$. 
Since $E_{p^m\ell^r}$ and $E_{\ell^r}$ are isogenous (\emph{via} the $p^m$th power Frobenius), their Mordell--Weil groups have the same rank. 
By the computation above, we thus have $\rk\big(E_{d_m}(K)\big)=R$ 
for all~$m\geq 0$.
\ProofEnd\end{proof}
 
\begin{rema} 
For an explicitly given prime $p$, it is rather easy to find primes  $\ell$ such that $p$ generates~$(\Z/\ell^2\Z)^\times$. 
Here is a sample for small values of $p$:
\renewcommand{\arraystretch}{1.2}
\[\begin{array}{c|l}
p& \{\ell \neq 2, p \text{ prime} : \langle p\bmod{\ell^2}\rangle = (\Z/\ell^2\Z)^\times\} \\ \hline 
3 & \{ 5, 7, 17, 19, 29, 31, 43, 53, 79, 89, 101, 113, 127, 137, 139, ... \} \\ 
5 & \{3, 7, 17, 23, 37, 43, 47, 53, 73, 83, 97, 103, 107, 113, 137, ... \} \\
7 & \{11, 13, 17, 23, 41, 61, 67, 71, 79, 89, 97, 101, 107, 127, 151, ... \} \\
11 & \{3, 13, 17, 23, 29, 31, 41, 47, 59, 67, 73, 101, 103, 109, 149, ... \} \\
13 & \{5, 11, 19, 31, 37, 41, 47, 59, 67, 71, 73, 83, 89, 97, 109, 137, ...\} \\ 
17 & \{5, 7, 11, 23, 31, 37, 41, 61, 97, 107, 113, 131, 139, 167, 173,...\} \\
\end{array} \]

Given a prime $p$, according to Artin's primitive root conjecture there should exist infinitely many primes~$\ell$ such that $p$ generates $(\Z/\ell\Z)^\times$ 
(better, the set of such primes should have a positive density).
Hooley \cite{Hooley_ArtinConj} has proved  Artin's conjecture under the assumption of GRH for the zeta function of certain number fields.
We also note that Heath-Brown later proved (unconditionally) that there are at most two exceptional primes~$p$ for which the conjecture fails to hold, see \cite{HeathBrown_Artin}.
Moreover, if $p$ generates $(\Z/\ell\Z)^\times$, one can show that $p$ also generates $(\Z/\ell^2\Z)^\times$ unless $p^{\ell-1}\equiv 1\bmod{\ell^2}$.
Experiments with small primes $p$ and $\ell$ suggest that the latter happens only very rarely (when $p=5$ for example, there are $30884$ primes $\ell$  in the range~$[3, 10^6]$ such that $5$ generates $(\Z/\ell\Z)^\times$, but only one with $5^{\ell-1}\equiv 1\bmod{\ell^2}$, namely $\ell= 40487$). 
It thus seems quite plausible that, for each prime $p\geq 3$, there are infinitely many primes $\ell$ such that $p$ generates $(\Z/\ell^2\Z)^\times$. 

Assuming that this is true, for all odd integers $R\geq 1$, one can exhibit infinitely many integers $(d_i)_{i\geq 1}$, \emph{all coprime to $p$}, such that $\rank(d_i)=R$ for all $i\geq 1$.
This would provide a `less artificial' construction of such a sequence than that given in the proof above.

\end{rema}

\noindent\hfill\rule{7cm}{0.5pt}\hfill\phantom{.}
 
 \paragraph{Acknowledgements} 
Part of this article is  based on a chapter in the author's PhD thesis \cite{Griffon_PHD}. 
The author wishes to thank his former advisor Marc Hindry for his guidance, and his fruitful and encouraging remarks.
Thanks are also due to Michael Tsfasman and Douglas Ulmer who made several useful comments on a preliminary version of this paper,
and to Peter Stevenhagen for interesting discussions about Artin's primitive root conjecture.
 
The writing of this paper was started at Universiteit Leiden, and  has been finished at Universit\"at Basel.
The author is grateful to both institutions for providing financial support and perfect working conditions; 
he is also partially supported by the ANR grant `FLAIR' (ANR-17-CE40-0012).


\newcommand{\mapolicebackref}[1]{\hspace*{-3pt}{\textcolor{orange}{\small$\uparrow$ #1}}}
\renewcommand*{\backref}[1]{\mapolicebackref{#1}}
\hypersetup{linkcolor=orange!80}

\small

\noindent\rule{7cm}{0.5pt}

\smallskip

{\sc Dep. Mathematik und Informatik, Universit\"at Basel}, Spiegelgasse 1, 4051 Basel, Switzerland.
 
{\it E-mail:} \href{mailto:richard.griffon@unibas.ch}{richard.griffon@unibas.ch}


\end{document}